\documentclass[10pt]{amsart}

\usepackage{amstext}
\usepackage{amssymb}
\usepackage{amsmath}
\usepackage{amsthm}
\usepackage{color}
\usepackage{multicol}
\usepackage{graphicx}
\graphicspath{{./pict/}}            
\DeclareGraphicsExtensions{.pdf,.png,.jpg}

\usepackage[english]{babel}

\theoremstyle{plain}
\newtheorem{theorem}{Theorem}[section]
\newtheorem{lemma}[theorem]{Lemma}
\newtheorem{prepos}[theorem]{Proposition}
\newtheorem{corol}[theorem]{Corollary}

\theoremstyle{definition}

\newtheorem{remark}[theorem]{Remark}

\newtheorem{conjecture}{Conjecture}

\def\sgn{\mathop{\rm sign}\nolimits}

\def\sgn{\mathop{\rm sign}\nolimits}

\def\sgn{\mathop{\rm sign}\nolimits}

\newcommand{\eqdef}{\stackrel{def}{=}}

\def\sgn{\mathop{\rm sign}\nolimits}

\begin{document}




\author[M.~Tyaglov]{Mikhail Tyaglov}

\address{School of Mathematical Sciences, Shanghai Jiao Tong University}
\email{tyaglov@sjtu.edu.cn}

\author[M.~Atia]{Mohamed J. Atia}

\address{Qassim university, College of Sciences, Buraydah, Saudi Arabia}
\email{jalel.atia@gmail.com}

\title[Non-real  zeroes of a homogeneous differential polynomial]{On the number of non-real  zeroes of a homogeneous differential polynomial
and a generalization of the Laguerre inequalities}



\begin{abstract}
Given a real polynomial $p$ with only real zeroes, we find upper and lower bounds for the number of non-real zeroes
of the differential polynomial
$$
F_{\varkappa}[p](z)\eqdef p(z)p''(z)-\varkappa[p'(z)]^2,
$$
where $\varkappa$ is a real number.

We also construct a counterexample to a conjecture by B.~Shapiro~\cite{Shapiro} on the number of real zeroes
of the polynomial $F_{\tfrac{n-1}{n}}[p](z)$ in the case when the real polynomial
$p$ of degree $n$ has non-real zeroes. We formulate some new conjectures
generalising the Hawaii conjecture.
\end{abstract}

\date{\small \today}

\maketitle

\vspace{2cm}

\setcounter{equation}{0}
\section{Introduction}

Let $p$ be a real polynomial with only real zeroes. Then
\begin{equation}\label{Laguere.inequality}
p(x)p''(x)-[p'(x)]^2\leqslant0,\qquad x\in\mathbb{R}.
\end{equation}
This inequality is called the \textit{Laguerre inequality}. It is well known that the entire functions of the Laguerre-P\'olya
class\footnote{The class of entire functions that are uniform limits on compact sets of sequences of polynomials with only real zeroes,
see~\cite{Polya_Schur,Obreschkoff} and references there.}
satisfy this inequality. The Laguerre inequality plays an important role in the study of distribution of
zeroes of real entire functions and in understanding the nature of the Riemann $\xi$-function and trigonometric
integrals, see~\cite{Craven_Csordas_89,Craven_Csordas_02},~\cite{Csordas_Escassut_05}--\cite{Csordas_Vishnyakova_13}
and references there for the generalizations of the Laguerre inequality, as well. The Laguerre inequality is sharp for
entire functions of the Laguerre-P\'olya class in the sense that
for any entire function $f$ in this class which is not a polynomial, the inequality
\begin{equation}\label{Laguere.ineqalities.generalized}
f(x)f''(x)-\varkappa[f'(x)]^2\leqslant0,\qquad x\in\mathbb{R},
\end{equation}
holds for $\varkappa\geqslant1$, and for any $\varkappa<1$ there exists a function in the Laguerre-P\'olya class which is not a polynomial such that~\eqref{Laguere.ineqalities.generalized} is not true for this function.
The function $e^{-x^2}$ is such an example for any $\varkappa<1$. In~\cite{Nicks_2012}, a lower bound was obtained for the number of non-real
zeroes of the function
$$
F_{\varkappa}[f](z)=f(z)f''(z)-\varkappa[f'(z)]^2,
$$
in the case when the entire function $f$ has only finitely many non-real zeroes. The zeroes of the
function~$F_{\varkappa}[f]$ when $f$ is a meromorphic function were studied in~\cite{Bergweiler_1995,Langley_1996,Langley_2011_1,Langley_2011_2,Tohge_1993}.

For polynomials with only real zeroes, inequality~\eqref{Laguere.inequality} is not sharp. In fact, if $p$ is a real
polynomial of degree $n$ with only real zeroes, then the following inequality holds~\cite{Love_1962,Niculescu_2000}:
\begin{equation}\label{Newton.inequality}
p(x)p''(x)-\dfrac{n-1}{n}[p'(x)]^2\leqslant0,\qquad x\in\mathbb{R}.
\end{equation}
This inequality is called the \textit{differential form of the Newton inequality}~\cite{Niculescu_2000}. According to~\cite{Shapiro}, this inequality (together with some other ones)
was found by G.\,P\'olya while he studied unpublished notes~of~J.\,Jensen.

For an arbitrary real polynomial  $p$, the Laguerre inequality~\eqref{Laguere.inequality} does not hold anymore, generally speaking. In~\cite{CCS_1987}, it was conjectured that in this case, the number of real zeroes of the function
\begin{equation}\label{main.function.Q.Hawaii}
Q_{1}[p](z)=\dfrac{p(z)p''(z)-[p'(z)]^2}{p^2(z)},
\end{equation}
does not exceed the number of non-real zeroes of the polynomial $p$. This conjecture was nicknamed the~\textit{Hawaii conjecture} by T.\,Sheil-Small~\cite{Sheil-Small_book}.
It was also noticed in~\cite{Csordas_2006} that the conjecture can be extended to entire functions. The zeroes of $Q_{1}[p](z)$ were also
studied in~\cite{Dilcher,Dilcher_Stol} as well as in~\cite{Borcea.Shapiro_2004,Edwards.Hinkkanen_2011}. It was believed that the Hawaii conjecture (if true) follows from some geometric properties of level curves of logarithmic derivatives, see e.g.~\cite{Borcea.Shapiro_2004}. However, it turned out (see~\cite{Tyaglov_Hawaii})
that the fact claiming by the conjecture is a non-trivial consequence of Rolle's theorem. Indeed, the long and sophisticated proof
is based on laborious calculations of the number of zeroes of $Q_1[p]$ on certain intervals.
Some researches still hope to find another proof, more simple than
the one given in~\cite{Tyaglov_Hawaii}.

Inspired by the Hawaii conjecture and the Newton inequalities, in~\cite{Shapiro} it was conjectured
that the number of real zeroes of the rational function
\begin{equation}\label{Newton.Q.function}
Q_{\tfrac{n-1}{n}}[p](z)=\dfrac{p(z)p''(z)-\frac{n-1}{n}[p'(z)]^2}{p^2(z)}
\end{equation}
does not exceed the number of non-real zeroes of the polynomial $p$ of degree $n$. The present work was initially motivated by this conjecture.
We construct a counter-example (see Section~\ref{section:non-real.zeroes}), and estimate the number of non-real zeroes of the differential
polynomial
\begin{equation}\label{main.polynomial.F}
F_{\varkappa}[p](z)\eqdef p(z)p''(z)-\varkappa[p'(z)]^2
\end{equation}
for $\varkappa\in\mathbb{R}$ and $p$ a real polynomial with \textit{only real} zeroes. Any multiple zero of $p$ is a zero of $F_{\varkappa}[p]$. The zeroes of the polynomial $F_{\varkappa}[p]$ which are zeroes of $p$ are called \textit{trivial} while all other zeroes are called \textit{non-trivial}. So if $p$ has only real zeroes, then all the non-real zeroes
of $F_{\varkappa}[p]$ are non-trivial.
In the present work, we find lower and upper bounds on the number of non-real zeroes of $F_{\varkappa}[p]$ for arbitrary \textit{real} $\varkappa$.
Note that if $\varkappa$ is a non-real number, then $F_{\varkappa}[p]$ \textit{has no} non-trivial zeroes at all by de Gua's rule~\cite{Polya_1930}, since in this case
any zero of $F_{\varkappa}[p]$ must be a zero of $p'$. Thus, the case of non-real~$\varkappa$ is trivial and is out of the scope of the present work.

The paper is organized as follows. In Section~\ref{section:main.results}, we state our main results on the number of non-real zeros of the
differential polynomial $F_{\varkappa}[p]$ in the case when $\varkappa$ is real and the polynomial $p$ has \textit{only real} zeros.
Section~\ref{section:function.Q} is devoted to the calculation of  the total number of non-trivial zeroes of the polynomial
$F_{\varkappa}[p]$ for arbitrary complex polynomial $p$. In Section~\ref{section:real.zeroes}, we prove our main results, inequalities~\eqref{estimates.0}--\eqref{estimates.4} stated in Section~\ref{section:main.results}. In Section~\ref{section:non-real.zeroes},
we consider the differential polynomial $F_{\varkappa}[p]$ for~$p$ to be an arbitrary real polynomial, and disprove
a conjecture of B.\,Shapiro~\cite{Shapiro} by a counterexample. We also provide a  conjecture generalizing
the Hawaii conjecture. In Appendix, we prove a generalization of an auxiliary fact established in~\cite[Lemma~2.5]{Tyaglov_Hawaii}.

Throughout the paper we use the following notation.

\vspace{2mm}

\noindent\textbf{Notation.} If $f$ is a real rational function or a real polynomial, by $Z_{C}(f)$ we denote the number of non-real zeroes of $f$, counting multiplicities, by $Z_{\mathbb{R}}(f)$ the number of real zeroes of $f$, counting multiplicities. In the sequel, we also denote the number of zeroes of $f$ in an interval $(a,b)$ and at a point $\alpha\in\mathbb{R}$
by $Z_{(a,b)}(f)$ and $Z_{\{\alpha\}}(f)$, respectively, thus $Z_{\mathbb{R}}(f)=Z_{(-\infty,+\infty)}(f)$ . Generally, the number of zeroes of $f$ on a set $X$ where $X$ is a subset of $\mathbb{R}$
will be denoted by $Z_{X}(f)$.

\setcounter{equation}{0}

\section{Main results}\label{section:main.results}

Let $p$ be a real polynomial with only real zeroes
\begin{equation}\label{main.polynomial.p.general}
p(z)=a_0\prod_{k=1}^{d}(z-\lambda_k)^{n_k}
\end{equation}
where $d$, $d\geqslant2$, is the number of \textit{distinct} zeroes of $p$, $n_k\in\mathbb{N}$ is the multiplicity of the zero $\lambda_k$ of $p$, $k=1,\ldots,d$, so the degree of $p$ is $n=n_1+...+n_d$, and we set
\begin{equation*}\label{order.zeroes.lambda}
\lambda_1<\lambda_2<\dots<\lambda_d.
\end{equation*}

\begin{theorem}\label{Theorem.main.inequalities.d.4}
Let $d\geqslant4$. Suppose that among the zeroes $\lambda_2,\ldots,\lambda_{d-1}$ there are $d_j$ zeroes of multiplicity~$m_j$, $j=1,\ldots,r$, $r\geqslant2$, so that
\begin{equation*}\label{order.multiplicities}
m_1<m_2<\cdots<m_r,
\end{equation*}
where $m_1=\min\{n_2,\ldots,n_{d-1}\}$ and $m_r=\max\{n_2,\ldots,n_{d-1}\}$, and
\begin{equation*}\label{formulas.miltiplicities}
\sum_{j=1}^{r}d_j=d-2,\qquad \sum_{j=1}^{r}m_j=n-n_1-n_d.
\end{equation*}
%

Then the following inequalities hold:
\begin{itemize}
\item[] if $\ \ \varkappa\leqslant\dfrac{m_1-1}{m_1}$, then
\begin{equation}\label{estimates.0}
Z_{C}(F_{\varkappa}[p])=0;
\end{equation}
\item[] if $\ \ \dfrac{m_j-1}{m_j}<\varkappa\leqslant\dfrac{m_{j+1}-1}{m_{j+1}}$, $j=1,\ldots,r-1$, then
\begin{equation}\label{estimates.1}
0\leqslant Z_{C}(F_{\varkappa}[p])\leqslant2\sum_{i=1}^{j}d_i;
\end{equation}
\item[] if $\ \ \dfrac{m_r-1}{m_r}<\varkappa<\dfrac{n-d+1}{n-d+2}$, then
\begin{equation}\label{estimates.2}
0\leqslant Z_{C}(F_{\varkappa}[p])\leqslant2d-4;
\end{equation}
\item[] if $\ \ \dfrac{n-d+k-1}{n-d+k}\leqslant\varkappa<\dfrac{n-d+k}{n-d+k+1}$, $k=2,\ldots,d-1$, then

\vspace{1mm}

\begin{equation}\label{estimates.3}
2k-2\leqslant Z_{C}(F_{\varkappa}[p])\leqslant2d-4;
\end{equation}
\item[] if $\ \ \varkappa=\dfrac{n-1}{n}$, then
\begin{equation}\label{estimates.3.5}
Z_{C}(F_{\varkappa}[p])=2d-4.
\end{equation}
\item[] if $\ \ \varkappa>\dfrac{n-1}{n}$, then
\begin{equation}\label{estimates.4}
Z_{C}(F_{\varkappa}[p])=2d-2.
\end{equation}
\end{itemize}
\end{theorem}
\noindent Theorem~\ref{Theorem.main.inequalities.d.4} follows from  Theorems~\ref{Theorem.number.complex.zeroes.n-1.n},~\ref{Theorem.upper.bound.1},  and~\ref{Theorem.lower.bound.1} established in Section~\ref{section:real.zeroes}.

\vspace{2mm}

Moreover, if $d\geqslant3$, and the zeroes $\lambda_2,\ldots,\lambda_{d-1}$ of $p$ defined in~\eqref{main.polynomial.p.general} are all of multiplicity $m$, then in the results
above we can formally set $m_1=m_r\equiv m$, so for $\varkappa\leqslant\dfrac{m-1}{m}$ the identity~\eqref{estimates.0} holds, while for $\varkappa>\dfrac{m-1}{m}$
we have inequalities~\eqref{estimates.2}--\eqref{estimates.4} (see Theorems~\ref{Theorem.number.complex.zeroes.n-1.n},~\ref{Theorem.upper.bound.2}, and~\ref{Theorem.lower.bound.1}).

\vspace{2mm}

Thus, the following theorem holds.
\begin{theorem}
Let $d\geqslant3$, and the zeroes $\lambda_2,\ldots,\lambda_{d-1}$ of the polynomial $p$ defined in~\eqref{main.polynomial.p.general} are all of multiplicity $m$, for $\varkappa\leqslant\dfrac{m-1}{m}$ one has $Z_{C}(F_{\varkappa}[p])=0$, while for $\varkappa>\dfrac{m-1}{m}$
inequalities~\eqref{estimates.2}--\eqref{estimates.4} hold $($with $m_r=m)$.
\end{theorem}

Finally, if $d=2$, that is, if the polynomial $p$ has only two distinct zeroes, then
Theorems~\ref{Theorem.number.complex.zeroes.n-1.n},~\ref{Theorem.real.zeroes.nonpositive.kappa},
and~\ref{Theorem.2.distinct.zeroes} imply the following fact.
\begin{theorem}
If the polynomial $p$ has exactly two distinct zeroes, then $Z_{C}(F_{\varkappa}[p])=0$ for $\varkappa\leqslant\dfrac{n-1}{n}$,
and $Z_{C}(F_{\varkappa}[p])=2$ for $\varkappa>\dfrac{n-1}{n}$ .
\end{theorem}

\begin{remark}
If the polynomial $p$ has a unique zero, then~$F_{\varkappa}[p]$
has no non-trivial zeroes, see~\eqref{Real.poly.with.multiple.root}--\eqref{F.for.unique.multiple.root}, so this case is trivial.
\end{remark}

Note that the non-trivial zeroes of the polynomial $F_{\varkappa}[p]$ are (not vice verse!) the solutions of the~equation
\begin{equation}\label{R.equation}
R(z)=\varkappa,
\end{equation}
where the function $R$ is defined as follows
\begin{equation*}\label{R.function}
R(z)\eqdef\dfrac{p(z)p''(z)}{[p'(z)]^2}.
\end{equation*}

If $p$ has only real zeroes and at least two of them are distinct, then the function $R$ is concave between its poles (Theorem~\ref{Lemma.concave.R}), and
inequalities~\eqref{estimates.0}--\eqref{estimates.4}, in fact,
show that $R$ has no maximum values over $\dfrac{n-2}{n-1}$ (inclusive), and can have at most one maximum value between
$\dfrac{n-3}{n-2}$ (inclusive) and $\dfrac{n-2}{n-1}$ (exclusive), at most two maximum values between
$\dfrac{n-4}{n-3}$ (inclusive) and $\dfrac{n-3}{n-2}$ (exclusive), etc., see~Fig.~\ref{pic.1}.

\begin{figure}[ht]
\centering \includegraphics[scale=0.4]{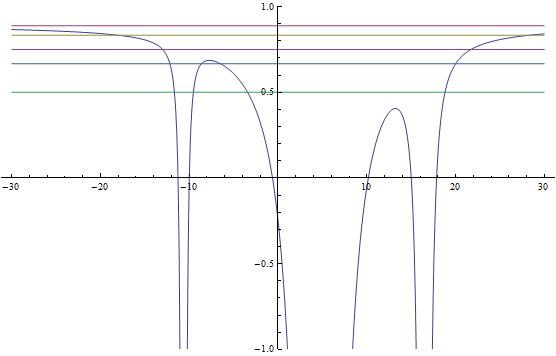} \caption{The function $R$ for the polynomial $p(z)=(z-15)(z+13)^4(x-20)^3(x+10)$.}
\label{pic.1}
\end{figure}

At the same time, not all the solutions of equation~\eqref{R.equation} are non-trivial zeroes of the polynomial~$F_{\varkappa}[p]$. But the non-real zeroes of this polynomial coincide with the non-real solutions of equation~\eqref{R.equation}, see~Theorem~\ref{Theorem.equations.equivalence} and Corollary~\ref{Corollary.equations.equivalence.excep}.
%

\setcounter{equation}{0}
\section{The total number of non-trivial zeroes of the polynomial $F_{\varkappa}[p]$}\label{section:function.Q}

Let $p$ be an arbitrary complex polynomial of degree $n$ 
\begin{equation}\label{Real.poly.main.poly.compl}
p(z)=a_0z^n+a_1z^{n-1}+a_2z^{n-2}+\cdots+a_n,\qquad n\geqslant2.
\end{equation}
Consider the differential polynomial
\begin{equation}\label{F.polynomial}
F_{\varkappa}[p](z)=p(z)p''(z)-\varkappa[p'(z)]^2.
\end{equation}
defined in~\eqref{main.polynomial.F}. As we mentioned in Introduction, all the zeroes of $F_{\varkappa}[p]$
that are not common with the polynomial $p$ are called non-trivial.

\vspace{2mm}

\noindent \textbf{Notation.} \textit{We denote the total number of the non-trivial
zeroes of $F_{\varkappa}[p]$ as $Z_{nt}(F_{\varkappa}[p])$.}

\vspace{2mm}

Suppose first that $\varkappa\neq\dfrac{n-1}{n}$. Then the polynomial $F_{\varkappa}[p]$ has exactly $2n-2$ zeroes, since
\begin{equation}\label{leading.coeff}
F_{\varkappa}[p](z)=a_0^2n^2\left(\dfrac{n-1}{n}-\varkappa\right)z^{2n-2}
+\cdots 
\end{equation}
so the leading coefficient of $F_{\varkappa}[p]$ is non-zero for $\varkappa\neq\dfrac{n-1}{n}$.

If the polynomial $p$ defined in~\eqref{Real.poly.main.poly.compl} has a unique zero $\lambda_1$ of multiplicity $n$,
\begin{equation}\label{Real.poly.with.multiple.root}
p(z)=a_0(z-\lambda_1)^n,
\end{equation}
where
\begin{equation}\label{Real.poly.with.multiple.root.2}
\lambda_1=-\dfrac{a_1}{na_0},
\end{equation}
then we have
\begin{equation}\label{F.for.unique.multiple.root}
F_{\varkappa}[p](z)=a_0^2n^2\left(\dfrac{n-1}{n}-\varkappa\right)(z-\lambda_1)^{2n-2},
\end{equation}
and $F_{\varkappa}[p](z)$ has no non-trivial zeroes.

Suppose that $p$ has \textit{at least two} distinct zeroes, and represent the polynomial $p$ in the following form
\begin{equation*}
p(z)=a_0\prod_{j=1}^{l_1}(z-\nu_j)\prod_{k=1}^{l_2}(z-\zeta_k)^{m_k},
\end{equation*}
where $n_j\geqslant2$, $j=1,\ldots,l_2$, so $p$ has $l_1$ simple zeroes and $l_2$ multiple zeroes.
We denote by $d$ the total number of \textit{distinct} zeroes of the polynomial $p$:
\begin{equation*}
d\eqdef l_1+l_2,
\end{equation*}
so
\begin{equation}\label{total.degree}
n=d+\sum_{k=1}^{l_2}(m_k-1).
\end{equation}
%

%
\begin{theorem}\label{Theorem.total.number.of.zeroes.F.general}
Let $p$ be a complex polynomial of degree $n$, $n\geqslant2$, with exactly $d$ \textbf{distinct} zeroes, $2\leqslant d\leqslant n$. Then
\begin{equation*}
Z_{nt}(F_{\varkappa}[p])=2d-2,
\end{equation*}
%
whenever $\varkappa\neq\dfrac{k-1}{k}$, $k=1,\ldots,n$.
\end{theorem}
\begin{proof}
If $\lambda$ is a zero of the polynomial $p$ of multiplicity $m$, then
\begin{equation}\label{Real.poly.multiple.zero}
p(z)=A(z-\lambda)^{m}+B(z-\lambda)^{m+1}+C(z-\lambda)^{m+2}+O\left((z-\lambda)^{m+3}\right), \quad A\neq0,\quad\text{as}\quad z\to\lambda,
\end{equation}
therefore,
\begin{equation}\label{Real.poly.multiple.zero.2.F}
\small
\begin{array}{l}
F_{\varkappa}[p](z)=\left(\dfrac{m-1}{m}-\varkappa\right)\cdot m^2\cdot A^2\cdot
(z-\lambda)^{2m-2}+2\left(\dfrac{m}{m+1}-\varkappa\right)\cdot m(m+1)\cdot AB\cdot(z-\lambda)^{2m-1}+\\
\\
+\left[\left(\dfrac{m}{m+1}-\varkappa\right)\cdot (m+1)^2\cdot B^2+2\left(\dfrac{m^2+m+1}{m(m+2)}-\varkappa\right)\cdot m(m+2)\cdot AC\right]\cdot(z-\lambda)^{2m}+O\left((z-\lambda)^{2m+1}\right),
\end{array}
\end{equation}
as $z\to\lambda$. Thus, if $\varkappa\neq\dfrac{k-1}{k}$, $k=1,\ldots,n$, then a zero $\lambda$ of $p$ of multiplicity $m>1$ is a trivial zero
of $F_{\varkappa}[p]$ of multiplicity $2m-2$, and is not a zero of $F_{\varkappa}[p]$ if $m=1$.

Now from~\eqref{total.degree}, we obtain that the total number of all the trivial zeroes of $F_{\varkappa}[p]$ is equal to
\begin{equation}\label{total.number.trivial.zeroes.F.kappa}
\sum\limits_{k=1}^{l_2}(2m_k-2)=2\sum\limits_{k=1}^{l_2}(m_k-1)=2n-2d.
\end{equation}
Therefore, the total number of all non-trivial zeroes of $F_{\varkappa}[p]$ equals $2d-2$,
since $\deg F_{\varkappa}[p]=2n-2$, as we established above.
\end{proof}

When  $\varkappa=\dfrac{k-1}{k}$, $k=1,\ldots,n$,  it is not an easy task to define the number of non-trivial zeroes of $F_{\varkappa}[p]$ for arbitrary complex polynomial $p$. However, one can estimate this number  if $p$ is a real polynomial with only real zeroes.

\begin{corol}\label{Corol.total.number.of.zeroes.F.general}
Let $p$ be a real polynomial of degree $n$, $n\geqslant2$, with only real zeroes If  $p$ has exactly $d$ distinct zeroes, $2\leqslant d\leqslant n$, and if
$\varkappa=\dfrac{k-1}{k}$ for some $k=2,\ldots,n-1$, then
\begin{equation}\label{non.trivial.zeros.F.except}
2d-2-2\alpha_k\leqslant Z_{nt}(F_{\varkappa}[p])\leqslant2d-2-\alpha_k,
\end{equation}
%
where $\alpha_k$ is the number of zeroes of the polynomial  $p$ of multiplicity $k$.

Moreover, if $k=1$, i.e. $\varkappa=0$, then
\begin{equation}\label{non.trivial.zeros.F.except.0}
2d-2\alpha_1\leqslant Z_{nt}(F_{\varkappa}[p])\leqslant2d-2-\alpha_1,
\end{equation}
where $\alpha_1$ is the number of simple zeros of the polynomial $p$.
\end{corol}

To prove this corollary we need to recall De Gua's rule.

\vspace{2mm}

\noindent\textbf{De Gua's rule (\cite{Polya_1930}, see also~\cite{Ki_Kim}).} \textit{If a real polynomial has only real zeroes, then its derivatives
have no multiple zeroes but the zeroes of the polynomial itself. Additionally, if a number $\xi$
is a (simple) real zero of $l^{th}$ derivative of the polynomial $p$, $l\geqslant1$, then
\begin{equation}\label{de.Gua.ineq.gen}
p^{(l-1)}(\xi)p^{(l+1)}(\xi)<0.
\end{equation}
}

\vspace{2mm}

\begin{proof}[Proof of Corollary~\ref{Corol.total.number.of.zeroes.F.general}]
From formul\ae~\eqref{Real.poly.multiple.zero}--\eqref{Real.poly.multiple.zero.2.F} it follows that if $\lambda$ is a zero of $p$
of multiplicity~$m$, then it is a zero of $F_{\varkappa}[p]$ of multiplicity at least $2m-2$. However, if $\varkappa=\dfrac{m-1}{m}$,
then $\lambda$ is a zero of $F_{\varkappa}[p]$ of multiplicity at least $2m-1$. If $B\neq0$ in~\eqref{Real.poly.multiple.zero}, then
its multiplicity is exactly $2m-1$, since in this case the coefficient at $(z-\lambda)^{2m-1}$ in~\eqref{Real.poly.multiple.zero.2.F} is non-zero. However, if $B=0$, which is equivalent
to the identity $p^{(m+1)}(\lambda)=0$, then $\lambda$ is a zero of $F_{\varkappa}[p]$ of multiplicity at least $2m$. In fact, its multiplicity
is exactly $2m$ in this case, since the coefficient at $(z-\lambda)^{2m}$ can be zero only if additionally $C=0$. But if it is so, then
we have
$$
p^{(m)}(\lambda)=m!A\neq0, \quad p^{(m+1)}(\lambda)=(m+1)!B=0, \quad p^{(m+2)}(\lambda)=(m+2)!C=0,
$$
which contradicts de Gua's rule (see also~\eqref{de.Gua.ineq}).

So, if $\varkappa=\dfrac{k-1}{k}$ for some $k=1,\ldots,n-1$, then a zero of the polynomial $p$ of multiplicity $k$ is a trivial zero of $F_{\varkappa}[p]$
of multiplicity $2k-1$ or $2k$. Any other zero of $p$ of multiplicity $m\neq k$ is a trivial zero $F_{\varkappa}[p]$ of multiplicity $2m-2$. Summing
all the trivial zeroes of $F_{\varkappa}[p]$ with their multiplicities and recalling that $\deg F_{\varkappa}[p]=2n-2$, we obtain the
inequalities~\eqref{non.trivial.zeros.F.except}.

If $\varkappa=0$, then we can improve the lower bound in~\eqref{non.trivial.zeros.F.except}, since the $p''$ has at most $\alpha_1 -2$ zeroes common with $p$.
\end{proof}

%

\begin{remark}
Note that the upper and lower bounds in Corollary~\ref{Corol.total.number.of.zeroes.F.general} are sharp.
Indeed, for
$$
p(z)=(z-1)^3(z-2)^2(z-3)^3,
$$
the polynomial $F_{\varkappa}[p]$ has exactly $2$ non-trivial zeroes, while for
$$
p(z)=(z-1)^3(z-2)^3(z-3)^3,
$$
the polynomial $F_{\varkappa}[p]$ has no non-trivial zeroes.
\end{remark}
\begin{remark}
The upper bound in Corollary~\ref{Corol.total.number.of.zeroes.F.general} is true for any complex polynomial. It can be proved in the
same way as in Corollary~\ref{Corol.total.number.of.zeroes.F.general}. The lower bound is much more difficult to establish, since de
Gua's rule is not applicable for arbitrary complex polynomial. However, in the present work we do need an analogue of Corollary~\ref{Corol.total.number.of.zeroes.F.general}
for arbitrary complex polynomial.
\end{remark}

\vspace{2mm}

Consider now the exceptional case $\varkappa=\dfrac{n-1}{n}$.
To calculate the total number of the non-trivial zeroes of the polynomial $F_{\tfrac{n-1}{n}}[p]$, we should define
its degree first.

If $p$ defined in~\eqref{Real.poly.main.poly.compl} has a unique multiple zero, then by~\eqref{F.for.unique.multiple.root}
we have $F_{\tfrac{n-1}{n}}[p](z)\equiv0$. Thus, to study the zeroes of the polynomial $F_{\varkappa}[p]$ we exclude such a situation in the sequel.

\begin{theorem}\label{Theorem.total.num.zeros.F.except}
Let the polynomial $p$ defined in~\eqref{Real.poly.main.poly.compl} have at least two distinct zeroes. Then the polynomial~$F_{\tfrac{n-1}{n}}[p]$
has exactly $2n-3-l$ zeroes for some $l$, $1\leqslant l\leqslant n-1$, if and only if $p$ has the form
\begin{equation}\label{Real.poly.extr.case}
p(z)=a_0\left(z+\dfrac{a_1}{a_0n}\right)^n+q(z),
\end{equation}
where $q$ is a polynomial of degree $n-l-1$.
\end{theorem}
\begin{proof}
Indeed, let the polynomial $p$ be of the form~\eqref{Real.poly.extr.case}. Denote by $b\neq0$ the leading coefficient of the polynomial $q$, and
\begin{equation*}
\lambda:=-\dfrac{a_1}{a_0n}.
\end{equation*}
Then we have
\begin{equation*}
\begin{array}{l}
\displaystyle F_{\tfrac{n-1}n}[p](z)=\left[a_0(z-\lambda)^n+q(z)\right]\cdot
\left[a_0n(n-1)(z-\lambda)^{n-2}+q''(z)\right]-\dfrac{n-1}{n}
\left[a_0n(z-\lambda)^{n-1}+q'(z)\right]^2=\\
\\
=a_0bl(l+1)z^{2n-l-3}+O(z^{2n-l-4})\quad\text{as}\quad |z|\to\infty,\\
\end{array}
\end{equation*}
as required.

\vspace{2mm}

Conversely,  for the polynomial $p$ defined in~\eqref{Real.poly.main.poly.compl} we have
\begin{equation*}
F_{\tfrac{n-1}{n}}[p](z)=2\left[a_0a_2-\binom{n}{2}\left(\dfrac{a_1}{n}\right)^2\right]z^{2n-4}+\cdots,
\end{equation*}
that is, the polynomial $F_{\tfrac{n-1}{n}}[p]$ has at most $2n-4$ zeroes
(unless it is not identically zero). In particular, for $n=2$, $F_{\tfrac{n-1}{n}}[p]$ is constant.

The leading coefficient of the polynomial $F_{\tfrac{n-1}{n}}[p]$ vanishes if and only if
\begin{equation*}\label{Real.poly.extr.case.coeff.poly.2}
a_2=\dfrac{\binom{n}{2}}{a_0}\left(\dfrac{a_1}{n}\right)^2.
\end{equation*}
In this case, we have
\begin{equation*}\label{Real.poly.extr.case.numerator.main.func.3}
F_{\tfrac{n-1}{n}}[p](z)=6\left[a_0a_3-\dfrac1{a_0}\cdot
\binom{n}{3}\left(\dfrac{a_1}{n}\right)^3\right]z^{2n-5}+\cdots.
\end{equation*}
The coefficient at the power $2n-5$ can also be equal to zero.

Continuing in such a way, we obtain that the polynomial $F_{\tfrac{n-1}{n}}[p]$ has the form
\begin{equation}\label{Real.poly.extr.case.numerator.main.func.general}
F_{\tfrac{n-1}{n}}[p](z)=l(l+1)\left[a_0a_{l+1}-\dfrac1{a_0^{l-1}}\cdot
\binom{n}{l+1}\left(\dfrac{a_1}{n}\right)^{l+1}\right]z^{2n-3-l}+\cdots,
\end{equation}
with $1\leqslant l\leqslant n-1$ if and only if the first $l$
coefficients of the polynomial $p$ satisfy the following
identities\footnote{The case $j=1$ is trivial but we include
it to formula~\eqref{Real.poly.extr.case.coeff.poly.general} for generality.}
\begin{equation}\label{Real.poly.extr.case.coeff.poly.general}
a_j=\dfrac{\binom{n}{j}}{a_0^{j-1}}\left(\dfrac{a_1}{n}\right)^j,\quad
j=1,2,\ldots,l.
\end{equation}

It easy to see now that the polynomial $p$ whose coefficients satisfy
the identities~\eqref{Real.poly.extr.case.coeff.poly.general} must have the form~\eqref{Real.poly.extr.case}.
\end{proof}

Note that if $l=n-1$, then by~\eqref{Real.poly.extr.case.numerator.main.func.general}
the polynomial $F_{\tfrac{n-1}{n}}[p]$ has at most $n-2$ zeroes. Moreover, it has less
than $n-2$ zeroes if and only if the coefficient $a_n$ satisfies
the identity
\begin{equation}\label{Real.poly.extr.case.coeff.poly.last}
a_n=\dfrac{\binom{n}{n}}{a_0^{n-1}}\left(\dfrac{a_1}{n}\right)^n.
\end{equation}

However, if the coefficients of the polynomial $p$ satisfy the
identities~\eqref{Real.poly.extr.case.coeff.poly.general} for $l=n-1$
and the identity~\eqref{Real.poly.extr.case.coeff.poly.last}, then the
polynomial~$p$ has the form~\eqref{Real.poly.with.multiple.root}--\eqref{Real.poly.with.multiple.root.2},
that is, it has a unique zero of multiplicity~$n$, so $F_{\tfrac{n-1}{n}}[p](z)\equiv0$ in this case as we mentioned above.

Now we are in a position to find the number of non-trivial zeroes of the polynomial $F_{\tfrac{n-1}{n}}[p]$.

\begin{theorem}\label{Theorem.total.number.of.zeroes.Q.exceptional}
Let $p$ be a complex polynomial of degree $n$, $n\geqslant2$, and let $d$ be the number of \textbf{distinct} zeroes of $p$, $2\leqslant d\leqslant n$. Suppose that the polynomial $p$ has the form~\eqref{Real.poly.extr.case}. Then
\begin{equation}\label{number.of.zeroes.Q.exceptional}
Z_{nt}\left(F_{\tfrac{n-1}{n}}[p]\right)=2d-3-l.
\end{equation}

If the polynomial $p$ has a unique multiple zero, then $F_{\tfrac{n-1}{n}}[p](z)\equiv0$.
\end{theorem}
\begin{proof}
By Theorem~\ref{Theorem.total.num.zeros.F.except}, the degree of the polynomial $F_{\tfrac{n-1}{n}}[p]$ equals $2n-3-l$, $1\leqslant l\leqslant n-1$, if and only if
the polynomial $p$ has the form~\eqref{Real.poly.extr.case}. In this case, the multiplicity of zeroes of $p$ is bounded by $n-l$. The number $2d-3-l$ can be obtained in the same way
as in the proof of Theorem~\ref{Theorem.total.number.of.zeroes.F.general}, see~\eqref{total.number.trivial.zeroes.F.kappa}.

If the polynomial $p$ has a unique multiple zero, then from~\eqref{F.for.unique.multiple.root} it follows that
$F_{\tfrac{n-1}{n}}[p](z)\equiv0$ as we mentioned above.
\end{proof}
\begin{remark}\label{Remark.excep}
We note that by the aforementioned de Gua's rule, Theorem~\ref{Theorem.total.number.of.zeroes.Q.exceptional} is not applicable for polynomials with only real zeroes for $l\geqslant2$. Indeed, if  $p$ is of the form~\eqref{Real.poly.extr.case}, then its $(n-l-1)^{th}$ derivative is as follows
$$
p^{(n-l-1)}(z)=\dfrac{a_0n!}{(l+1)!!}\left(z-\dfrac{a_1}{a_0n}\right)^{l+1}+C,
$$
where $C$ is a non-zero constant. It is clear now that for $l\geqslant2$ the polynomial $p$ does not satisfy the inequality~\eqref{de.Gua.ineq.gen} for zeroes of the $(n-l)^{th}$ derivative of $p$, so it cannot have only real zeroes.
\end{remark}

The proof of Theorem~\ref{Theorem.total.number.of.zeroes.Q.exceptional} implies the following curious fact
on the lower bound for the number of distinct zeroes of a polynomial.
\begin{corol}
Let $p$ be a complex polynomial of
the form
\begin{equation*}
p(z)=a_0(z-\lambda)^n+q(z),\qquad n\geqslant2,
\end{equation*}
where $\deg q=k-1$ for some $k$, $1\leqslant k\leqslant n-1$. Then the number $d$ of \textbf{distinct} zeroes of $p$ satisfies the inequality
\begin{equation}\label{lower.bound.distinct.zeroes.p}
d\geqslant\left\lfloor\dfrac{n-k}{2}\right\rfloor+2\geqslant2.
\end{equation}
%
\end{corol}
\begin{proof}
According to~\eqref{number.of.zeroes.Q.exceptional}, the number $2d-3-(n-k)$ is nonnegative, so we have
\begin{equation*}
d\geqslant\dfrac{n-k+3}{2}.
\end{equation*}
Since $d$ is integer, inequality~\eqref{lower.bound.distinct.zeroes.p} holds.

\end{proof}

In the sequel we use the following auxiliary rational function
\begin{equation}\label{main.function.Q}
Q_{\varkappa}[p](z)=\dfrac{p(z)p''(z)-\varkappa[p'(z)]^2}{p^2(z)}.
\end{equation}

It is easy to see that the set of all non-trivial zeroes of $F_{\varkappa}[p]$ coincides with the set of all zeroes of~$Q_{\varkappa}[p]$.
\begin{remark}\label{Remark.unique.zero}
If the polynomial $p$ has a unique zero $\lambda_1$ of multiplicity $n$, then $Q_{\varkappa}[p](z)=\dfrac{C}{(z-\lambda_1)^2}$,
where the constant $C$ equals zero if and only if $\varkappa=\dfrac{n-1}{n}$.
\end{remark}

\setcounter{equation}{0}
\section{The number of non-real zeroes of the polynomial $F_{\varkappa}[p]$ when $p$ has only real zeroes}\label{section:real.zeroes}

In this section, we estimate the number of non-real zeroes of the polynomial $F_{\varkappa}[p]$ defined in~\eqref{F.polynomial}
provided~$p$ is a real polynomial of degree $n$, $n\geqslant2$, with only real zeroes.

Let
\begin{equation}\label{Lemma.concave.R.proof.1}
p(z)=a_0\prod\limits_{k=1}^d(z-\lambda_k)^{n_k},\qquad\sum\limits_{k=1}^dn_k=n,\qquad a_0>0,\ d\leqslant n.
\end{equation}
Then
\begin{equation}\label{Lemma.concave.R.proof.2}
p'(z)=a_0n\prod\limits_{k=1}^d(z-\lambda_k)^{n_k-1}\cdot\prod\limits_{j=1}^{d-1}(z-\mu_j),
\end{equation}
where we fix the order of the zeroes  indexing as follows
\begin{equation}\label{Zeroes.ineq}
\lambda_1<\mu_1<\lambda_2<\cdots<\lambda_{d-1}<\mu_{d-1}<\lambda_d.
\end{equation}

The simplicity of the zeroes $\mu_j$, $j=1,\ldots, d-1$, of the polynomial $p'$ is guaranteed by de Gua's rule. By the same rule, we have
\begin{equation}\label{de.Gua.ineq}
p(\mu_k)p''(\mu_k)<0,\qquad k=1,\ldots,d-1.
\end{equation}
%

The following auxiliary lemma will be of use in the sequel.
\begin{theorem}\label{Lemma.concave.R}
If a polynomial $p$ of degree $n$, $d\geqslant2$, has only real zeroes,
then the rational function
\begin{equation}\label{function.R}
R(z)\eqdef\dfrac{p(z)p''(z)}{[p'(z)]^2}
\end{equation}
is concave between its poles and has the form
\begin{equation}\label{function.R.simple.frac}
R(z)=\dfrac{n-1}{n}+\sum\limits_{j=1}^{d-1}\dfrac{\beta_j}{(z-\mu_j)^2},
\end{equation}
where
\begin{equation}\label{Residues2.function.R}
\beta_j=\dfrac{p(\mu_j)}{p''(\mu_j)}<0,\qquad j=1,\ldots,d-1.
\end{equation}
Here $d$ is the number of \textbf{distinct} zeroes of $p$, and $\mu_j$, $j=1,\ldots,d-1$, are simple distinct
zeroes of~$p'$ such that $p(\mu_j)\neq0$, $j=1,\ldots,d-1$. Moreover, the function $R$ is concave between its poles.
\end{theorem}
\begin{proof}
Consider a polynomial $p$ as in~\eqref{Lemma.concave.R.proof.1}, and note that $\deg [pp'']=\deg[(p')^2]$, so $\lim\limits_{|z|\to\infty}R(z)$ is equal to
the ratio of the leading coefficients of the polynomials $pp''$ and $(p')^2$. Since the leading
coefficient of $pp''$ equals $a_0^2n(n-1)$, and the leading coefficient of $(p')^2$ is $a_0^2n^2$,
we have
$$
\lim\limits_{|z|\to\infty}R(z)=\dfrac{a_0^2n(n-1)}{a_0^2n^2}=\dfrac{n-1}{n}.
$$
Furthermore, it is clear that the numbers $\mu_j$, $j=1,\ldots,d-1$ are the only poles of the function $R$, and
$p(\mu_j)p''(\mu_j)\neq0$ by~\eqref{de.Gua.ineq}. Thus, $R$ has the form
\begin{equation}\label{Lemma.concave.R.proof.3}
R(z)=\dfrac{n-1}{n}+\sum\limits_{j=1}^{d-1}\left(\dfrac{\alpha_j}{z-\mu_j}+\dfrac{\beta_j}{(z-\mu_j)^2}\right).
\end{equation}

However, it is easy to see that $\mu_j$ defined in~\eqref{Lemma.concave.R.proof.2}--\eqref{Zeroes.ineq} is a simple pole of the function $\dfrac{p}{p'}$ with the residue $p(\mu_j)/p''(\mu_j)$:
$$
\dfrac{p(z)}{p'(z)}=\dfrac{p(\mu_j)/p''(\mu_j)}{z-\mu_j}+O(1)\qquad\text{as}\qquad z\to\mu_j,
$$
so
$$
R(z)=1-\left(\dfrac{p(z)}{p'(z)}\right)'=
\dfrac{p(\mu_j)/p''(\mu_j)}{(z-\mu_j)^2}+O(1)\qquad\text{as}\qquad z\to\mu_j.
$$
Consequently, in~\eqref{Lemma.concave.R.proof.3} the coefficients
$\alpha_j$ are all zero, and the coefficients $\beta_j$ are defined by formula~\eqref{Residues2.function.R}, and their negativity follows
from~\eqref{de.Gua.ineq}.

Finally, the function $R''$ has the form
$$
R''(z)=\sum\limits_{j=1}^{d-1}\dfrac{3!\beta_{j}}{(z-\mu_j)^4},
$$
so it is negative at any real point where it exists, hence on the real line $R$ is concave between its poles.
\end{proof}

From formul\ae~\eqref{function.R}--\eqref{function.R.simple.frac} it easy to
describe the location of zeroes of the function $R$.

\begin{theorem}\label{Theorem.R.zeroes}
Let a polynomial $p$ of degree $n\geqslant2$ with only real zeroes and its derivative $p'$ be defined as in~\eqref{Lemma.concave.R.proof.1}--\eqref{Zeroes.ineq}. Then the function $R$ defined in~\eqref{function.R} has exactly one zero (counting multiplicities) on each of the intervals $(-\infty,\mu_1)$ and $(\mu_{d-1},+\infty)$, and exactly two zeroes (counting multiplicities) on each of the intervals $(\mu_j,\mu_{j+1})$, $j=1,\ldots,d-2$.
\end{theorem}
\begin{proof}
From~\eqref{function.R.simple.frac} it follows that
\begin{equation}\label{R.at.infinity}
R(x)\to\dfrac{n-1}{n}\quad\text{as}\quad x\to\pm\infty.
\end{equation}
Moreover, the function
$$
R'(z)=-\sum\limits_{j=1}^{d-1}\dfrac{2!\beta_j}{(z-\mu_j)^3}
$$
is decreasing between its poles (its derivative $R''$ is negative on $\mathbb{R}$ where it exists), and $R'(x)\to\pm0$ as $x\to\pm\infty$.
Consequently, $R'(x)<0$ in the interval $(-\infty,\mu_1)$, and $R'(x)>0$ in the interval $(\mu_{d-1},+\infty)$. (We remind the reader that
the zeroes of $p$ and $p'$ are indexed in the order~\eqref{Zeroes.ineq}.)
Thus, the function~$R$ decreases from $\dfrac{n-1}{n}$ to $-\infty$ on $(-\infty,\mu_1)$, and
increases from $-\infty$ to $\dfrac{n-1}{n}$ on $(\mu_{d-1},+\infty)$.
%
%

\vspace{2mm}

The monotone behaviour of the function $R$ on the intervals $(-\infty,\mu_1)$ and $(\mu_{d-1},+\infty)$
shows that~$R(z)$ has exactly one zero, counting multiplicities, on each of these intervals.
Furthermore, since $p$ has only real zeroes, the function $R$ has exactly $2d-2$ real zeroes
(possibly multiple), since $pp''$ has exactly $2n-2$ zeroes, $2n-2d$ of which are common with $(p')^2$.
The concavity of $R$ between its poles implies that $R$ has exactly $2$ zeroes, counting multiplicities, in each interval $(\mu_k,\mu_{k+1})$, $k=1,\ldots,d-2$, so $R$ has $2d-4$ zeroes in the interval $(\mu_1,\mu_{d-1})$.
\end{proof}

The most important property of the function $R$ is represented by the following theorem.
\begin{theorem}\label{Theorem.equations.equivalence}
Given a real polynomial $p$ with only real zeroes and $d\geqslant 2$, for any $\varkappa\in\mathbb{R}$, $\varkappa\neq\dfrac{n-1}{n}$, the equation
\begin{equation}\label{R.kappa.equation}
R(z)=\varkappa,
\end{equation}
has exactly $2d-2$ solutions. Moreover, if $\varkappa\neq\dfrac{m-1}{m}$, $m=1,\ldots,n-1$, then the set of solutions of  equation~\eqref{R.kappa.equation} coincides with the set of non-trivial zeroes of $F_{\varkappa}[p]$.

If $\varkappa=\dfrac{m-1}{m}$ for certain $m$, $1\leqslant m\leqslant n-1$, then the set of solutions of   equation~\eqref{R.kappa.equation} coincides with the set of non-trivial zeroes of $F_{\varkappa}[p]$, except the zeroes of the polynomial $p$ of multiplicity $m$.

In particular, the set of non-real zeroes of the polynomial $F_{\varkappa}[p]$ defined in~\eqref{F.polynomial} coincides with the set of
non-real solutions of equation~\eqref{R.kappa.equation}.
\end{theorem}
\begin{proof}
It is clear that equation~\eqref{R.kappa.equation} is equivalent to the equation
\begin{equation*}\label{Theorem.equations.equivalence.proof.1}
\dfrac{F_{\varkappa}[p](z)}{[p'(z)]^2}=0.
\end{equation*}
If $\lambda$ is a zero of $p$ of multiplicity $m\geqslant2$, then by~\eqref{Real.poly.multiple.zero.2.F} it is a zero of the polynomial $F_{\varkappa}[p]$ of multiplicity \textit{at~least} $2m-2$ and is a zero of $(p')^2$ of multiplicity exactly $2m-2$. Thus, from~\eqref{total.degree} it follows that the total number of common zeroes of $F_{\varkappa}[p]$ and $(p')^2$, counting multiplicities, equals $2n-2d$ for any $\varkappa\in\mathbb{R}$ (including $\varkappa=\dfrac{n-1}{n}$). Since for any $\varkappa\in\mathbb{R}$, $\varkappa\neq\dfrac{n-1}{n}$, the total number of zeroes of $F_{\varkappa}[p]$ is $2n-2$ by~\eqref{leading.coeff}, we have that the total number of solutions of equation~\eqref{R.kappa.equation} equals $2d-2$.

Now let us notice the following simple fact. If $\lambda_k$ is a zero
of $p$ of multiplicity $n_k\geqslant1$, then from~\eqref{Real.poly.multiple.zero}, it follows that
\begin{equation}\label{R.at.multiple.zero.expansion}
\begin{array}{l}
R(z)=\dfrac{n_{k}-1}{n_{k}}+\dfrac{2B}{An_{k}^2}(z-\lambda_{k})+\\
\\
+\dfrac{3\left[(n_{k}+1)B^2+2A\cdot C\cdot n_{k}\right]}{A^2n_{k}^3}(z-\lambda_{k})^2+O\left((z-\lambda_{k})^3\right)\qquad\text{as}\quad z\to\lambda_k,
\end{array}
\end{equation}
therefore,
\begin{equation}\label{R.at.multiple.zero}
R(\lambda_k)=\dfrac{n_k-1}{n_k}.
\end{equation}

Consequently, if $\varkappa\neq\dfrac{m-1}{m}$, $m=1,\ldots,n-1$, then
the zeroes of the polynomial $p$ do not solve equation~\eqref{R.kappa.equation}. So the set of zeroes of $F_{\varkappa}[p]$, and the set of solutions of~\eqref{R.kappa.equation} coincide.

If $\varkappa=\dfrac{m-1}{m}$ for certain $m$, $1\leqslant m\leqslant n-1$, then any zero of $p$ of multiplicity $m$ is a solution to equation~\eqref{R.kappa.equation}. Such solutions are trivial zeroes of the polynomial $F_{\varkappa}[p]$. However, all other solutions (including all non-real solutions) of equation~\eqref{R.kappa.equation} are non-trivial zeroes of~$F_{\varkappa}[p]$, and only these.
\end{proof}

The exceptional case $\varkappa=\dfrac{n-1}{n}$ can be treated in a similar way with Theorem~\ref{Theorem.total.num.zeros.F.except} for\footnote{When $l>1$ in Theorem~\ref{Theorem.total.num.zeros.F.except}, the polynomial has non-real zeroes by de Gua's rule, so we exclude such a case from our investigation, see Remark~\ref{Remark.excep}} $l=1$.
\begin{corol}\label{Corollary.equations.equivalence.excep}
Given a real polynomial $p$ with only real zeroes and $d\geqslant 2$, for  $\varkappa=\dfrac{n-1}{n}$, equation~\eqref{R.kappa.equation}
has exactly $2d-4$ solutions. Moreover, the set of non-real zeroes of the polynomial $F_{\tfrac{n-1}{n}}[p]$ defined in~\eqref{F.polynomial} coincides with the set of
non-real solutions of equation~\eqref{R.kappa.equation}.
\end{corol}

Thus, in what follows we count the number of non-real solutions of equation~\eqref{R.kappa.equation} or the number of non-real zeroes of the function $Q_{\varkappa}[p]$ that coincide with the number~$Z_C(F_{\varkappa}[p])$.

\vspace{5mm}

Now we are in a position to consider various intervals for the real parameter $\varkappa$.

\subsection{The cases $\varkappa\geqslant\dfrac{n-1}{n}$ and $\varkappa\leqslant0$.}

Formul\ae~\eqref{function.R.simple.frac}--\eqref{Residues2.function.R} imply that
\begin{equation}\label{Newton.inequality.via.R}
R(x)<\dfrac{n-1}{n}
\end{equation}
for any $x\in\mathbb{R}$ (as $x$ approaches a pole of $R(x)$, it tends to $-\infty$), so
equation~\eqref{R.kappa.equation}
has no real solutions~for
\begin{equation*}
\varkappa\geqslant\dfrac{n-1}{n}.
\end{equation*}
Thus, Theorem~\ref{Theorem.equations.equivalence} and Corollary~\ref{Corollary.equations.equivalence.excep} imply the following fact.
\begin{theorem}\label{Theorem.number.complex.zeroes.n-1.n}
Let $p$ be a real polynomial of degree $n$ with only real zeroes, and $d\geqslant2$, where $d$ is the number of \textbf{distinct} zeroes of the polynomial $p$. Then
\begin{equation}\label{Theorem.number.complex.zeroes.n-1.n.formula.1}
Z_{C}(F_{\varkappa}[p])=2d-2\qquad\text{for}\quad \varkappa>\dfrac{n-1}{n},
\end{equation}
and
\begin{equation}\label{Theorem.number.complex.zeroes.n-1.n.formula.2}
Z_{C}(F_{\varkappa}[p])=2d-4\qquad\text{for}\quad \varkappa=\dfrac{n-1}{n}.
\end{equation}
\end{theorem}
Note that inequality~\eqref{Newton.inequality.via.R} is equivalent to the Newton inequality~\eqref{Newton.inequality} for polynomials with only real zeroes. Moreover, it is clear that
\begin{equation}\label{Newton.inequality.Q}
Q_{\varkappa}[p](x)<0,\qquad \varkappa\geqslant\dfrac{n-1}{n}
\end{equation}
for any $x\in\mathbb{R}$ where $Q_{\varkappa}[p](x)$ is finite.


\vspace{5mm}

Let now $\varkappa\leqslant0$. The following theorem shows that all the zeroes of $F_{\varkappa}[p]$ are real, in this case.
\begin{theorem}\label{Theorem.real.zeroes.nonpositive.kappa}
Let a real polynomial $p$ of degree $n$ have only real zeroes, and let $d\geqslant2$, where $d$ is the number of \textbf{distinct} zeroes of the polynomial $p$.

If $\varkappa\leqslant0$,
then
\begin{equation}\label{Q.neg.kappa.num.zeroes}
Z_{C}(F_{\varkappa}[p])=0. 
\end{equation}
Moreover, all the (real) non-trivial zeroes of $F_{\varkappa}[p]$ are simple for any $\varkappa<0$.
\end{theorem}
\begin{proof}
The theorem asserts that all the non-trivial zeroes of polynomial $F_{\varkappa}[p]$ are real if $\varkappa\leqslant0$ and $p$
has only real zeroes. By de Gua's rule, a number is a non-trivial zero of the polynomial $F_{\varkappa}[p]$ if and only if
it is a solution of equation~\eqref{R.kappa.equation} provided $\varkappa\leqslant0$.

From Theorem~\ref{Theorem.R.zeroes},  the function $R$ has no non-real zeroes, and neither does the polynomial $F_{\varkappa}[p]$ for $\varkappa=0$ according to~Theorem~\ref{Theorem.equations.equivalence}. Consequently,
$$
Z_{C}(F_{0}[p])=0.
$$

Consider now $\varkappa<0$. By Theorem~\ref{Theorem.R.zeroes}, the function $R(z)$ has exactly one zero (counting multiplicities) in each of
the intervals $(-\infty,\mu_1)$ and $(\mu_{d-1},+\infty)$, and exactly two zeroes (counting multiplicities) in each of the intervals $(\mu_k,\mu_{k+1})$, $k=1,\ldots,d-2$.
Let us denote by $\xi_1^{(k)}$ and $\xi_2^{(k)}$, $\xi_1^{(k)}\leqslant\xi_2^{(k)}$, the zeroes of~$R$ on $(\mu_k,\mu_{k+1})$, $k=1,\ldots,d-2$.
And let $\xi_0$ and $\xi_{d-1}$ be the zeroes of~$R$ on the intervals $(-\infty,\mu_1)$ and $(\mu_{d-1},+\infty)$, respectively.
From~\eqref{function.R.simple.frac}--\eqref{Residues2.function.R} it follows that $R(x)$ monotonously decreases to $-\infty$ as $x\to\pm\mu_j$.
Therefore, for $\varkappa<0$, the equation $R(x)=\varkappa$ has exactly one solution, counting multiplicity,
in each interval $(\xi_0,\mu_1)$, $(\mu_{d-1},\xi_{d-1})$, $(\mu_k,\xi_1^{(k)})$, $(\xi_2^{(k)},\mu_{k+1})$, $k=1,\ldots,d-2$, and no solutions
on the intervals $(-\infty,\xi_0)$, $(\xi_{d-1},+\infty)$, and $(\xi_1^{(k)},\xi_2^{(k)})$ for  $k=2,\ldots,d-2$.
So it has exactly $2d-2$ real \textit{simple} solutions.

Thus,  for $\varkappa<0$ all solutions of the
equation $R(x)=\varkappa$ are real and simple. By Theorem~\ref{Theorem.equations.equivalence}, the polynomial $F_{\varkappa}[p]$
has no non-real zeroes for $\varkappa<0$, as required.
\end{proof}
In summary, if $\varkappa\leqslant0$, then all the non-trivial zeroes of the polynomial $F_{\varkappa}[p]$
are real while all of them are non-real for $\varkappa\geqslant\dfrac{n-1}{n}$ whenever the polynomial $p$ of degree $n$ has only real zeroes.
So the number of non-real zeroes of
$Q_{\varkappa}[p]$ must increase as~$\varkappa$ changes continuously from $0$ to $\dfrac{n-1}{n}$.

Additionally, we found out that the function $R$ has exactly one local maximum on each
interval $(\mu_k,\mu_{k+1})$, $k=1,\ldots,d-2$, and the values of these maxima are on the interval $\left[0,\dfrac{n-1}{n}\right)$. Thus, if~$\varkappa$ increases from $0$ to $\dfrac{n-1}{n}$ and becomes larger than some local maximum of $R$, then equation~\eqref{R.kappa.equation} loses a pair of real solutions.

\subsection{The case $0<\varkappa<\dfrac{n-1}{n}$.}

Now we are in a position to estimate the number of non-real zeroes of the polynomial $F_{\varkappa}[p]$ for $0<\varkappa<\dfrac{n-1}{n}$.

\vspace{1mm}

First, we prove the following simple auxiliary fact.
\begin{lemma}\label{Lemma.min.numb.sol.R}
Let $p$ be a real polynomial of degree $n\geqslant2$,  $d\geqslant2$.  Then
equation~\eqref{R.kappa.equation} has exactly one solution (counting multiplicities) on each of the intervals $(-\infty,\mu_1)$ and $(\mu_{d-1},+\infty)$, provided $p$ is real-rooted and $0<\varkappa<\dfrac{n-1}{n}$.
\end{lemma}
\begin{proof}
In the proof of Theorem~\ref{Theorem.R.zeroes}, we established that the function $R$ decreases from $\dfrac{n-1}{n}$ to $-\infty$ on $(-\infty,\mu_1)$, and
increases from $-\infty$ to $\dfrac{n-1}{n}$ in $(\mu_{d-1},+\infty)$ whenever $p$ has only real zeroes. So the equation $R(x)=\varkappa$
has exactly one solution, counting multiplicities, on each interval $(-\infty,\mu_1)$ and $(\mu_{d-1},+\infty)$ for any
$0<\varkappa<\dfrac{n-1}{n}$.
\end{proof}
Let us now consider the case $d=2$. Recall that by $d$ we denote the number of distinct zeroes of $p$, see~\eqref{Lemma.concave.R.proof.1}.
\begin{theorem}\label{Theorem.2.distinct.zeroes}
Let the polynomial $p$ have two distinct zeroes, and let all zeroes of $p$ be real. Then
\begin{equation*}
Z_{C}(F_{\varkappa}[p])=0,
\end{equation*}
whenever $0<\varkappa<\dfrac{n-1}{n}$.
\end{theorem}
\begin{proof}
According to Theorem~\ref{Theorem.equations.equivalence}, all the non-real zeroes of the polynomial $F_{\varkappa}[p]$ are solutions to equation~\eqref{R.kappa.equation}. By the same theorem,
 this equation has exactly $2d-2$ solutions if $\varkappa\neq\dfrac{n-1}n$. Since $d=2$, we obtain that equation~\eqref{R.kappa.equation} has exactly $2$ solutions. At the same time, Lemma~\ref{Lemma.min.numb.sol.R} guaranties that equation~\eqref{R.kappa.equation} has at least $2$ real solutions. Consequently, if $d=2$, equation~\eqref{R.kappa.equation} has no non-real solutions. Therefore, $F_{\varkappa}[p]$ has no non-real zeroes in this case.
\end{proof}
Thus, the case $d=2$ is completely covered by Theorems~\ref{Theorem.number.complex.zeroes.n-1.n}, \ref{Theorem.real.zeroes.nonpositive.kappa}, and~\ref{Theorem.2.distinct.zeroes}, and we deal with the case $d\geqslant3$ in the rest of the present section.


\vspace{3mm}

Let again the polynomial $p$ and its derivative be given by~\eqref{Lemma.concave.R.proof.1}--\eqref{Zeroes.ineq}. We will distinguish the following two cases.
\begin{itemize}\label{cases}
\item[1)]\label{cases.1} $d\geqslant4$, and among the zeroes $\lambda_2,\ldots,\lambda_{d-1}$ we have $d_j$ zeroes of multiplicity $m_j$, $j=1,\ldots,r$, $r\geqslant2$, such that
\begin{equation*}
m_1<m_2<\cdots<m_r,
\end{equation*}
where $m_1=\min\{n_2,\ldots,n_{d-1}\}$ and $m_r=\max\{n_2,\ldots,n_{d-1}\}$, and
\begin{equation*}
\sum_{j=1}^{r}d_j=d-2,\qquad \sum_{j=1}^{r}m_j=n-n_1-n_2,
\end{equation*}
\item[2)]\label{cases.2} $d\geqslant3$, and all the zeroes $\lambda_2,\ldots,\lambda_{d-1}$ are of multiplicity $m=\max\{n_2,\ldots,n_{d-1}\}=\min\{n_2,\ldots,n_{d-1}\}$.
\end{itemize}
The last case can be treated as the case when $m=m_1=m_r$, or as the case of $r=1$.

\vspace{3mm}

For the case $1)$ we have the following the upper bound of the number  of non-real roots of the polynomial~$F_{\varkappa}[p]$
for $0<\varkappa<\dfrac{n-1}{n}$.
\begin{theorem}\label{Theorem.upper.bound.1}
Let $p$ and $p'$ be defined as in~\eqref{Lemma.concave.R.proof.1}--\eqref{Zeroes.ineq}. Let also $d\geqslant4$, and suppose that
among the zeroes $\lambda_2,\ldots,\lambda_{d-1}$ we have $d_j$ zeroes of multiplicity $m_j$, $j=1,\ldots,r$, $r\geqslant2$, such that
\begin{equation}\label{order.multiplicities.1}
m_1<m_2<\cdots<m_r,
\end{equation}
where $m_1=\min\{n_2,\ldots,n_{d-1}\}$ and $m_r=\max\{n_2,\ldots,n_{d-1}\}$, and
\begin{equation}\label{formulas.miltiplicities.1}
\sum_{j=1}^{r}d_j=d-2,\qquad \sum_{j=1}^{r}m_j=n-n_1-n_2.
\end{equation}
Then the following inequalities hold:
\begin{itemize}
\item[] If $0<\varkappa\leqslant\dfrac{m_1-1}{m_1}$, then
\begin{equation}\label{Theorem.upper.bound.1.estimates.0}
Z_{C}(F_{\varkappa}[p])=0;
\end{equation}
\item[] If $\dfrac{m_j-1}{m_j}<\varkappa\leqslant\dfrac{m_{j+1}-1}{m_{j+1}}$, $j=1,\ldots,r-1$, then
\begin{equation}\label{Theorem.upper.bound.1.estimates.1}
Z_{C}(F_{\varkappa}[p])\leqslant2\sum_{i=1}^{j}d_i;
\end{equation}
\item[] If $\dfrac{m_r-1}{m_r}<\varkappa<\dfrac{n-1}{n}$, then
\begin{equation}\label{Theorem.upper.bound.1.estimates.2}
Z_{C}(F_{\varkappa}[p])\leqslant2d-4.
\end{equation}
\end{itemize}
\end{theorem}
\begin{proof}
By Theorem~\ref{Theorem.equations.equivalence}, equation~\eqref{R.kappa.equation} has
exactly $2d-2$ solutions, and the set all the non-real solutions to this equation coincides
with the set of all non-real zeroes of the polynomials $F_{\varkappa}[p]$ whose number
of non-trivial zeroes is at most $2d-2$ according to Theorem~\ref{Theorem.total.number.of.zeroes.F.general} and
Corollary~\ref{Corol.total.number.of.zeroes.F.general}. Since equation~\eqref{R.kappa.equation}
has at least $2$ real solutions by Lemma~\ref{Lemma.min.numb.sol.R}, we have
\begin{equation}\label{Theorem.upper.bound.1.estimates.00}
Z_{C}(F_{\varkappa}[p])\leqslant2d-4\qquad\text{for}\quad\varkappa<\dfrac{n-1}{n}.
\end{equation}
However, this inequality can be improved  for some values of $\varkappa$.

Indeed, from~\eqref{R.at.multiple.zero} it follows that if $\lambda$ is a zero of the
polynomial $p$ of multiplicity $m_i\geqslant2$, then  for any $\varkappa\leqslant\dfrac{m_i-1}{m_i}$,
the equation $R(x)=\varkappa$ has exactly two solutions (counting multiplicities),
on the interval $(\mu_{k-1},\mu_k)$ containing~$\lambda$, since $R$ is concave
between its poles by Theorem~\ref{Lemma.concave.R}. So if $\varkappa\leqslant\dfrac{m_i-1}{m_i}$
for certain $m_i$, $i=2,\ldots,r$, defined in~\eqref{order.multiplicities.1}, then
the equation $R(z)=\varkappa$ has \textit{at least}
\begin{equation*}
2+2\sum\limits_{j=i}^{r}d_j
\end{equation*}
real solutions (counting multiplicities). Therefore, equation~\eqref{R.kappa.equation}
has \textit{at most}
\begin{equation*}
2d-2-2-2\sum\limits_{j=i}^{r}d_j=2\sum\limits_{j=1}^{i-1}d_j
\end{equation*}
non-real solutions by~\eqref{formulas.miltiplicities.1}. So inequalities~\eqref{Theorem.upper.bound.1.estimates.1}
are true, since the set of non-real solution of equation~\eqref{R.kappa.equation} coincides
with the set of non-real zeroes of the polynomial~$F_{\varkappa}[p]$ according to
Theorem~\ref{Theorem.equations.equivalence}.

If $0<\varkappa\leqslant\dfrac{m_1-1}{m_1}$, then the equation $R(z)=\varkappa$
has exactly two zeroes in every interval $(\mu_{k-1},\mu_k)$, $k=1,\ldots,d-1$,
due to concavity of $R$ and by~\eqref{R.at.multiple.zero}. Consequently,
all the solutions of this equation are real, so the identity~\eqref{Theorem.upper.bound.1.estimates.0} is true.
\end{proof}
In the same way, one can prove the corresponding result for the case $2)$.
\begin{theorem}\label{Theorem.upper.bound.2}
Let $p$ and $p'$ be defined as in~\eqref{Lemma.concave.R.proof.1}--\eqref{Zeroes.ineq}. Let also $d\geqslant3$, and suppose that
all the zeroes $\lambda_2,\ldots,\lambda_{d-1}$ are of multiplicity $m=\max\{n_2,\ldots,n_{d-1}\}=\min\{n_2,\ldots,n_{d-1}\}$.

Then the following holds:
\begin{itemize}
\item[] If $0<\varkappa\leqslant\dfrac{m-1}{m}$, then
\begin{equation*}\label{Theorem.upper.bound.2.estimates.0}
Z_{C}(F_{\varkappa})=0;
\end{equation*}
\item[] If $\dfrac{m-1}{m}<\varkappa<\dfrac{n-1}{n}$, then
\begin{equation*}\label{Theorem.upper.bound.2.estimates.2}
Z_{C}(F_{\varkappa})\leqslant2d-4.
\end{equation*}
\end{itemize}
\end{theorem}

Thus, the upper bound for the number of non-real zeroes of the polynomial~$F_{\varkappa}[p]$
is established for any~$\varkappa\in\mathbb{R}$.

\vspace{3mm}

In what follows, we find the lower bound for the number of non-real zeroes of the polynomial~$F_{\varkappa}[p]$. To do this, we
estimate from above the number of real zeroes of the auxiliary rational function $Q_{\varkappa}[p]$ defined in~\eqref{main.function.Q}.
As we mentioned above, the set of zeroes of $Q_{\varkappa}[p]$ coincides with the set of all non-trivial zeroes of $F_{\varkappa}[p]$.

Together with $Q_{\varkappa}[p]$, let us consider the function
\begin{equation}\label{Q.hat}
\widehat{Q}_{\varkappa}[p](z)=Q_{2-\tfrac1{\varkappa}}[p'](z)=\dfrac{p'(z)p'''(z)-\left(2-\frac1{\varkappa}\right)[p''(z)]^2}{[p'(z)]^2}.
\end{equation}

Relation between the number of zeroes of the functions $Q_{\varkappa}[p]$ and $\widehat{Q}_{\varkappa}[p]$ on an interval is provided by the following proposition.
\begin{prepos}\label{lemma.55}
Let $p$ be a real polynomial, $\varkappa>0$, and $a,b\in\mathbb{R}$.
If
$p(z)\neq0$, $p'(z)\neq0$, and $p''(z)\neq0$ for
$z\in(a,b)$, then
\begin{equation}\label{main.work.formula.11}
Z_{(a,b)}(Q_{\varkappa}[p])\leqslant1+Z_{(a,b)}\left(\widehat{Q}_{\varkappa}[p]\right).
\end{equation}
\end{prepos}
For the case $\varkappa=1$, this fact was proved in~\cite[Lemma~2.5]{Tyaglov_Hawaii}. The proof of Proposition~\ref{lemma.55} is the
same as the proof of Lemma~2.5 in~\cite{Tyaglov_Hawaii}, so we skip the proof here but provide it in Appendix for completeness
(see Theorem~\ref{lemma.5}).

If $p$ has only real zeroes, then the following consequence
of inequality~\eqref{main.work.formula.11} is true.
\begin{theorem}\label{Theorem.ineq.Q.Q.hat.greater.1/2}
Let $p$ be a real polynomial of degree $n$ with real zeroes, and $d\geqslant3$. Then
\begin{equation}\label{auxiliary.main.ineq}
Z_{\mathbb{R}}(Q_{\varkappa}[p])\leqslant Z_{\mathbb{R}}\left(\widehat{Q}_{\varkappa}[p]\right),
\end{equation}
for any $\dfrac12<\varkappa<\dfrac{n-1}{n}$.
\end{theorem}

\noindent We split the proof of this theorem into a few lemmas.

Let the polynomial $p$ and its derivative be defined in~\eqref{Lemma.concave.R.proof.1}--\eqref{Lemma.concave.R.proof.2}.
We fix the order of the zeroes indexing as in~\eqref{Zeroes.ineq}:
\begin{equation*}\label{Zeroes.ineq.1}
\lambda_1<\mu_1<\lambda_2<\cdots<\lambda_{d-1}<\mu_{d-1}<\lambda_d,
\end{equation*}

Consider the following auxiliary functions
\begin{equation}\label{R.R.functions}
R[p](z)\eqdef \dfrac{p(z)p''(z)}{[p'(z)]^2},\quad\qquad R[p'](z)\eqdef \dfrac{p'(z)p'''(z)}{[p''(z)]^2}.
\end{equation}
By Theorem~\ref{Theorem.equations.equivalence}, the sets of non-real zeroes of the functions $Q_{\varkappa}[p]$ and $\widehat{Q}_{\varkappa}[p]$ coincide with the sets of non-real solutions
of the equations
\begin{equation}\label{Theorem.ineq.Q.Q.hat.greater.1/2.proof.0}
R[p](z)=\varkappa,\qquad\text{and}\qquad R[p'](z)=2-\dfrac1{\varkappa},
\end{equation}
respectively, since the sets of all zeroes of $Q_{\varkappa}[p]$ and $\widehat{Q}_{\varkappa}[p]$ coincide the sets of all non-trivial zeroes of the polynomials  $F_{\varkappa}[p]$ and  $F_{2-\tfrac1{\varkappa}}[p']$ , respectively.

\vspace{1mm}

If $\lambda$ is a (real) zero of the polynomial $p$ of multiplicity $m\geqslant2$, then it is a zero of $p'$ of multiplicity
$m-1$. Moreover, if $\varkappa=\dfrac{m-1}{m}$, then
$2-\dfrac1{\varkappa}=\dfrac{m-2}{m-1}$. Thus, if $\varkappa\neq\dfrac{m-1}{m}$, then the sets of zeroes of the functions $Q_{\varkappa}[p]$ and $\widehat{Q}_{\varkappa}[p]$ coincide with the sets of all solutions of the
equations~\eqref{Theorem.ineq.Q.Q.hat.greater.1/2.proof.0}, respectively.
\begin{lemma}\label{Lemma.main.ends}
Let $p$ and $p'$ be defined as in~\eqref{Lemma.concave.R.proof.1}--\eqref{Zeroes.ineq}. For the intervals $(-\infty,\mu_1]$ and $[\mu_{d-1},+\infty)$, the following inequalities hold
\begin{equation}\label{auxiliary.main.ineq.ends}
Z_{(-\infty,\mu_1]}(Q_{\varkappa}[p])\leqslant Z_{(-\infty,\mu_1]}\left(\widehat{Q}_{\varkappa}[p]\right)
\quad\text{and}\quad Z_{[\mu_{d-1},+\infty)}(Q_{\varkappa}[p])\leqslant Z_{[\mu_{d-1},+\infty)}\left(\widehat{Q}_{\varkappa}[p]\right)
\end{equation}
for any $\dfrac12<\varkappa<\dfrac{n-1}{n}$.
\end{lemma}
\begin{proof}
Consider the interval $(-\infty,\mu_1]$. If $\lambda_1$ is a simple zero of $p$, then $p'(\lambda_1)\neq0$, and $p''(z)\neq0$ for all $z\in(-\infty,\mu_1)$, and $\mu_1$ is a simple zero of $p'$. By Theorems~\ref{Theorem.R.zeroes} and~\ref{Theorem.equations.equivalence},
the function~$Q_{\varkappa}[p]$  has a unique simple zero on $(-\infty,\mu_1]$ for any $\dfrac12<\varkappa<\dfrac{n-1}{n}$, that is,
\begin{equation}\label{Theorem.ineq.Q.Q.hat.greater.1/2.proof.1}
Z_{(-\infty,\mu_1]}\left(Q_{\varkappa}[p]\right)=1.
\end{equation}
Moreover, since $R[p]$ is monotone decreasing on $(-\infty,\mu_1)$ and $R[p](\lambda_1)=0$, the function~$Q_{\varkappa}[p]$ has
a unique simple zero on the interval $(-\infty,\lambda_1)$ and no zeroes on $(\lambda_1,\mu_1]$. Analogously, we conclude
that~$\widehat{Q}_{\varkappa}[p]$ has a simple real zero on the interval $(-\infty,\mu_1]$:
\begin{equation}\label{Theorem.ineq.Q.Q.hat.greater.1/2.proof.1.1}
Z_{(-\infty,\mu_1]}\left(\widehat{Q}_{\varkappa}[p]\right)=1
\end{equation}
since $\mu_1$ is a zero of $R[p']$.

Suppose now that $\lambda_1$ is a zero of $p$ of multiplicity $m$,
$2\leqslant m\leqslant n-1$. Then the polynomial $p''$ has a simple zero $\gamma_0$ on the interval $(\lambda_1,\mu_1)$
which is a pole of the function $R[p']$. If $\dfrac{m-1}{m}<\varkappa<\dfrac{n-1}{n}$, so that,
$\dfrac{m-2}{m-1}<2-\dfrac1{\varkappa}<\dfrac{n-2}{n-1}$, then the function~$Q_{\varkappa}[p]$ has a unique
simple zero on the interval $(-\infty,\lambda_1)$ and no zeroes on $[\lambda_1,\mu_1]$ by
Theorems~\ref{Theorem.R.zeroes}--\ref{Theorem.equations.equivalence} and by~\eqref{R.at.multiple.zero}.
The same argument proves that $\widehat{Q}_{\varkappa}[p]$ has a unique simple zero on $(-\infty,\lambda_1)$
and no zeroes on $[\lambda_1,\gamma_0)$.

If $\dfrac12<\varkappa<\dfrac{m-1}{m}$, so that, $0<2-\dfrac1{\varkappa}<\dfrac{m-2}{m-1}$, then the
function~$Q_{\varkappa}[p]$ has a unique simple zero on the interval $[\lambda_1,\mu_1]$ and
no zeroes on $(-\infty,\lambda_1)$ by Theorems~\ref{Theorem.R.zeroes}--\ref{Theorem.equations.equivalence}
and by~\eqref{R.at.multiple.zero}. Analogously, we have that $\widehat{Q}_{\varkappa}[p]$ has a unique
simple zero on $[\lambda_1,\gamma_0)$ and no zeroes on $(-\infty,\lambda_1)$.

Thus, if $\lambda_1$ is a zero of $p$ of multiplicity $m$, $2\leqslant m\leqslant n-1$ and
$\varkappa\neq\dfrac{m-1}{m}$, then the following holds
\begin{equation}\label{Theorem.ineq.Q.Q.hat.greater.1/2.proof.1.3.pre}
1=Z_{(-\infty,\mu_1]}(Q_{\varkappa}[p])\leqslant Z_{(-\infty,\mu_1]}\left(\widehat{Q}_{\varkappa}[p]\right),
\end{equation}
since the function~$\widehat{Q}_{\varkappa}[p]$ can have zeroes
on the interval $[\gamma_0,\mu_1]$.

Finally, if $\lambda_1$ is a zero of $p$ of multiplicity $m$, $3\leqslant m\leqslant n-1$ and
$\varkappa=\dfrac{m-1}{m}$, so that $2-\dfrac1{\varkappa}=\dfrac{m-2}{m-1}$, then $\lambda_1$
is a solution to both equations~\eqref{Theorem.ineq.Q.Q.hat.greater.1/2.proof.0}, and it is
not a zero of the functions $Q_{\varkappa}[p]$ and $\widehat{Q}_{\varkappa}[p]$. So by~\eqref{R.at.multiple.zero} we have
\begin{equation}\label{Theorem.ineq.Q.Q.hat.greater.1/2.proof.1.2}
Z_{(-\infty,\mu_1]}(Q_{\varkappa}[p])=Z_{(-\infty,\gamma_0]}\left(\widehat{Q}_{\varkappa}[p]\right)=0,
\quad\text{and}\quad Z_{(-\infty,\mu_1]}\left(\widehat{Q}_{\varkappa}[p]\right)\geqslant0.
\end{equation}
Thus, from~\eqref{Theorem.ineq.Q.Q.hat.greater.1/2.proof.1}--\eqref{Theorem.ineq.Q.Q.hat.greater.1/2.proof.1.2} we have
\begin{equation}\label{Theorem.ineq.Q.Q.hat.greater.1/2.proof.1.3}
Z_{(-\infty,\mu_1]}(Q_{\varkappa}[p])\leqslant Z_{(-\infty,\mu_1]}\left(\widehat{Q}_{\varkappa}[p]\right)
\end{equation}
for $\dfrac12<\varkappa<\dfrac{n-1}{n}$.

\vspace{1mm}

In the same way, one can prove that
\begin{equation}\label{Theorem.ineq.Q.Q.hat.greater.1/2.proof.2}
Z_{[\mu_{d-1},+\infty)}(Q_{\varkappa}[p])\leqslant Z_{[\mu_{d-1},+\infty)}\left(\widehat{Q}_{\varkappa}[p]\right)
\end{equation}
for $\dfrac12<\varkappa<\dfrac{n-1}{n}$, as required.
\end{proof}

The next lemma deals with the intervals $(\mu_k,\mu_{k+1})$.
\begin{lemma}\label{Lemma.main.mids}
Let $p$ and $p'$ be defined as in~\eqref{Lemma.concave.R.proof.1}--\eqref{Zeroes.ineq}. For any interval $(\mu_k,\mu_{k+1})$, $k=1,\ldots,d-2$, the following inequality holds
\begin{equation}\label{auxiliary.main.ineq.mids}
Z_{(\mu_k,\mu_{k+1})}(Q_{\varkappa}[p])\leqslant Z_{(\mu_k,\mu_{k+1})}\left(\widehat{Q}_{\varkappa}[p]\right)
\end{equation}
for $\dfrac12<\varkappa<\dfrac{n-1}{n}$.
\end{lemma}
\begin{proof}
Since $\dfrac12<\varkappa<\dfrac{n-1}{n}$, one has $0<2-\dfrac1{\varkappa}<\dfrac{n-2}{n-1}$. By~\eqref{Zeroes.ineq} the polynomial $p$ has a unique zero~$\lambda_{k+1}$ on the interval $(\mu_{k},\mu_{k+1})$.

Suppose that $\lambda_{k+1}$ is a simple zero of $p$. Then $p''$ has a unique simple zero $\gamma_k$ on the interval $(\mu_k,\mu_{k+1})$. Without loss of generality we may assume\footnote{The case $\gamma_k>\lambda_{k+1}$ can be considered in the same way.} that $\gamma_k\leqslant\lambda_{k+1}$. Since $\lambda_{k+1}$ is a simple zero of $p$, by Theorem~\ref{Theorem.equations.equivalence} the zeroes of $Q_{\varkappa}[p]$ and $\widehat{Q}_{\varkappa}[p]$ on the interval $(\mu_k,\mu_{k+1})$ coincide with solutions of equations~\eqref{Theorem.ineq.Q.Q.hat.greater.1/2.proof.0}, respectively.

The numbers $\gamma_k$ and $\lambda_{k+1}$ are zeroes of the function $R[p]$
on the interval $(\mu_k,\mu_{k+1})$, so Theorems~\ref{Lemma.concave.R} and~\ref{Theorem.equations.equivalence} imply
\begin{equation}\label{Theorem.ineq.Q.Q.hat.greater.1/2.proof.4.1}
0=Z_{(\mu_k,\gamma_k]}(Q_{\varkappa}[p])\leqslant Z_{(\mu_k,\gamma_k]}\left(\widehat{Q}_{\varkappa}[p]\right),
\end{equation}
and
\begin{equation}\label{Theorem.ineq.Q.Q.hat.greater.1/2.proof.4.2}
Z_{[\lambda_{k+1},\mu_{k+1})}(Q_{\varkappa}[p])=0
\end{equation}
for $\dfrac12<\varkappa<\dfrac{n-1}{n}$.

\vspace{1mm}

Suppose first that $\gamma_{k}<\lambda_{k+1}$. Then $Z_{(\gamma_k,\lambda_{k+1})}(Q_{\varkappa})$ equals $0$ or $2$, since $R[p]$ is concave on $(\mu_k,\mu_{k+1})$ by Theorem~\ref{Lemma.concave.R}. Moreover,
by Theorems~\ref{Lemma.concave.R} and~\ref{Theorem.equations.equivalence}  the function $\widehat{Q}_{\varkappa}[p]$ has an even number of zeroes (at most two) on each interval $(\mu_k,\gamma_k)$ and $(\gamma_k,\mu_{k+1})$, since $R[p'](\mu_k)=R[p'](\mu_k)=0$ according to~\eqref{R.R.functions}.

 If $Z_{(\gamma_k,\lambda_{k+1})}\left(\widehat{Q}_{\varkappa}[p]\right)=1$ and $Z_{[\lambda_{k+1},\mu_{k+1})}\left(\widehat{Q}_{\varkappa}[p]\right)=1$, then
from~\eqref{Theorem.ineq.Q.Q.hat.greater.1/2.proof.4.2} we get
\begin{equation}\label{Theorem.ineq.Q.Q.hat.greater.1/2.proof.6.1}
0=Z_{[\lambda_{k+1},\mu_{k+1})}(Q_{\varkappa}[p])\leqslant-1+Z_{[\lambda_{k+1},\mu_{k+1})}\left(\widehat{Q}_{\varkappa}[p]\right)=0,
\end{equation}
and from~\eqref{main.work.formula.11}
\begin{equation}\label{Theorem.ineq.Q.Q.hat.greater.1/2.proof.6.2}
Z_{(\gamma_k,\lambda_{k+1})}(Q_{\varkappa}[p])\leqslant1+Z_{(\gamma_k,\lambda_{k+1})}\left(\widehat{Q}_{\varkappa}[p]\right).
\end{equation}

If $Z_{(\gamma_k,\lambda_{k+1})}\left(\widehat{Q}_{\varkappa}[p]\right)=0$ and $Z_{[\lambda_{k+1},\mu_{k+1})}\left(\widehat{Q}_{\varkappa}[p]\right)=2$, then from~\eqref{Theorem.ineq.Q.Q.hat.greater.1/2.proof.4.2} it follows that
\begin{equation}\label{Theorem.ineq.Q.Q.hat.greater.1/2.proof.6.3}
0=Z_{[\lambda_{k+1},\mu_{k+1})}(Q_{\varkappa}[p])\leqslant Z_{[\lambda_{k+1},\mu_{k+1})}\left(\widehat{Q}_{\varkappa}[p]\right)=2,
\end{equation}
and by~\eqref{main.work.formula.11} we have
\begin{equation}\label{Theorem.ineq.Q.Q.hat.greater.1/2.proof.6.4}
0=Z_{(\gamma_k,\lambda_{k+1})}(Q_{\varkappa}[p])\leqslant Z_{(\gamma_k,\lambda_{k+1})}\left(\widehat{Q}_{\varkappa}[p]\right)=0,
\end{equation}
since $Z_{(\gamma_k,\lambda_{k+1})}(Q_{\varkappa}[p])$ is an even number as we mentioned above.

If $Z_{(\gamma_k,\lambda_{k+1})}\left(\widehat{Q}_{\varkappa}[p]\right)=2$ and $Z_{[\lambda_{k+1},\mu_{k+1})}\left(\widehat{Q}_{\varkappa}[p]\right)=0$, then
from~\eqref{main.work.formula.11} we obtain
\begin{equation}\label{Theorem.ineq.Q.Q.hat.greater.1/2.proof.6.5}
Z_{(\gamma_k,\lambda_{k+1})}(Q_{\varkappa}[p])\leqslant Z_{(\gamma_k,\lambda_{k+1})}\left(\widehat{Q}_{\varkappa}[p]\right)=2,
\end{equation}
since  $Z_{(\gamma_k,\lambda_{k+1})}(Q_{\varkappa}[p])$ is an even number.
We also have
\begin{equation}\label{Theorem.ineq.Q.Q.hat.greater.1/2.proof.6.6}
0=Z_{[\lambda_{k+1},\mu_{k+1})}(Q_{\varkappa}[p])\leqslant Z_{[\lambda_{k+1},\mu_{k+1})}\left(\widehat{Q}_{\varkappa}[p]\right)=0.
\end{equation}

Finally, if $Z_{(\gamma_k,\lambda_{k+1})}\left(\widehat{Q}_{\varkappa}[p]\right)=0$ and $Z_{[\lambda_{k+1},\mu_{k+1})}\left(\widehat{Q}_{\varkappa}[p]\right)=0$,
then inequality~\eqref{main.work.formula.11} and identity~\eqref{Theorem.ineq.Q.Q.hat.greater.1/2.proof.4.2} imply
\begin{equation}\label{Theorem.ineq.Q.Q.hat.greater.1/2.proof.6.7}
0=Z_{[\gamma_{k},\mu_{k+1})}(Q_{\varkappa}[p])=Z_{[\gamma_{k},\mu_{k+1})}\left(\widehat{Q}_{\varkappa}[p]\right).
\end{equation}
Thus, from \eqref{Theorem.ineq.Q.Q.hat.greater.1/2.proof.4.1} and~\eqref{Theorem.ineq.Q.Q.hat.greater.1/2.proof.6.1}--\eqref{Theorem.ineq.Q.Q.hat.greater.1/2.proof.6.7}
we obtain inequality~\eqref{auxiliary.main.ineq.mids} for $\gamma_k<\lambda_{k+1}$ and~$\dfrac12<\varkappa<\dfrac{n-1}{n}$.

If $\lambda_{k+1}=\gamma_k$, then the equation $R(z)=\varkappa$ has no solutions on the interval $(\mu_{k},\mu_{k+1})$
for $\varkappa>0$, so inequality~\eqref{auxiliary.main.ineq.mids} is true in this case. Consequently, if $\lambda_{k+1}$ is a simple zero of the polynomial $p$, then inequality~\eqref{auxiliary.main.ineq.mids} holds for  $\dfrac12<\varkappa<\dfrac{n-1}{n}$.

\vspace{3mm}

Let now $\lambda_{k+1}$ be a zero of $p$ of multiplicity $n_{k+1}\geqslant2$. Then $p'(\lambda_{k+1})=0$,
so $p''$ has a unique simple zero at each of the intervals $(\mu_k,\lambda_{k+1})$ and $(\lambda_{k+1},\mu_{k+1})$.
We denote these zeroes as $\gamma_k'$ and $\gamma_{k}''$, respectively, so
\begin{equation*}
\mu_{k}<\gamma_{k}'<\lambda_{k+1}<\gamma_k''<\mu_{k+1}.
\end{equation*}
It is clear that $R[p](\gamma_k')=R[p](\gamma_{k}'')=0$. By Theorem~\ref{Lemma.concave.R}, the function $R[p]$ is concave (as well as the function~$R[p']$), so
the equation $R[p](x)=\varkappa$ has no solutions on the
intervals $(\mu_k,\gamma_{k}']$ and $[\gamma_k'',\mu_{k+1})$, and an even number of solution (at most two) on the interval $(\gamma_k',\gamma_k'')$ for any $\dfrac12<\varkappa<\dfrac{n-1}{n}$. Consequently, by Theorem~\ref{Theorem.equations.equivalence} we have
\begin{equation}\label{Theorem.ineq.Q.Q.hat.greater.1/2.proof.8.1}
0=Z_{(\mu_k,\gamma_{k}']}(Q_{\varkappa}[p])\leqslant Z_{(\mu_k,\gamma_{k}']}\left(\widehat{Q}_{\varkappa}[p]\right),\qquad 0=Z_{[\gamma_{k}'',\mu_{k+1})}(Q_{\varkappa}[p])\leqslant Z_{[\gamma_{k}'',\mu_{k+1})}\left(\widehat{Q}_{\varkappa}[p]\right).
\end{equation}

On the interval $(\gamma_{k}',\gamma_{k}'')$, the function $Q_{\varkappa}[p]$
has 0 or 2 zeroes if $\varkappa\neq\dfrac{n_{k+1}-1}{n_{k+1}}$ by Theorem~\ref{Theorem.equations.equivalence}. However, if $\varkappa=\dfrac{n_{k+1}-1}{n_{k+1}}$, then the equation $R[p](x)=\varkappa$ has
two solutions (counting multiplicities) on the interval $(\gamma_{k}',\gamma_{k}'')$, and
at least one of the solutions is always $\lambda_{k+1}$ by~\eqref{R.at.multiple.zero}. But $\lambda_{k+1}$ is not a zero of the function~$Q_{\varkappa}[p]$. Thus, $Q_{\varkappa}[p]$ has at most one (counting multiplicities) zero on the interval $(\gamma_k',\gamma_k'')$ in this case.

Let $n_{k+1}\geqslant2$, and $\dfrac{n_{k+1}-1}{n_{k+1}}<\varkappa<\dfrac{n-1}{n}$, so that $\dfrac{n_{k+1}-2}{n_{k+1}-1}<2-\dfrac1{\varkappa}<\dfrac{n-2}{n-1}$. Then the zeroes of $Q_{\varkappa}[p]$ and~$Q_{\varkappa}[p']$ on the interval $(\gamma_k',\gamma_k'')$ coincide with the solutions of equations~\eqref{Theorem.ineq.Q.Q.hat.greater.1/2.proof.0}, respectively.  Due to concavity of the functions $R[p]$ and $R[p']$ on the interval $(\gamma_k',\gamma_k'')$, the functions $Q_{\varkappa}[p]$ and~$Q_{\varkappa}[p']$ have even number of zeroes on this interval and at most two. Moreover, each of these functions has no zeroes on one of the intervals $(\gamma_k',\lambda_{k+1}]$ and $[\lambda_{k+1},\gamma_k'')$.

There are possible two situations:
\begin{itemize}
\item [1)] $0=Z_{(\gamma_k',\lambda_{k+1}]}\left(\widehat{Q}_{\varkappa}[p]\right)$. Then by Proposition~\ref{lemma.55}  we have
\begin{equation}\label{Theorem.ineq.Q.Q.hat.greater.1/2.proof.8.1.0}
0=Z_{(\gamma_k',\lambda_{k+1}]}(Q_{\varkappa}[p])\leqslant Z_{(\gamma_k',\lambda_{k+1}]}\left(\widehat{Q}_{\varkappa}[p]\right)=0,
\end{equation}
since the number $Z_{(\gamma_k',\lambda_{k+1}]}(Q_{\varkappa}[p])$ can be 0 or 2 only.
It is clear now that the numbers
$Z_{(\lambda_{k+1},\gamma_k'')}(Q_{\varkappa}[p])$ and $Z_{(\lambda_{k+1},\gamma_k'')}\left(\widehat{Q}_{\varkappa}[p]\right)$
are even (at most two), so by Proposition~\ref{lemma.55}, one has
\begin{equation}\label{Theorem.ineq.Q.Q.hat.greater.1/2.proof.8.1.1}
Z_{(\lambda_{k+1},\gamma_k'')}(Q_{\varkappa}[p])\leqslant
Z_{(\lambda_{k+1},\gamma_k'')}\left(\widehat{Q}_{\varkappa}[p]\right),
\end{equation}
so from~\eqref{Theorem.ineq.Q.Q.hat.greater.1/2.proof.8.1.0}--\eqref{Theorem.ineq.Q.Q.hat.greater.1/2.proof.8.1.1} it follows that
\begin{equation}\label{Theorem.ineq.Q.Q.hat.greater.1/2.proof.8.1.2}
Z_{(\gamma_k',\gamma_k'')}(Q_{\varkappa}[p])\leqslant
Z_{(\gamma_{k}',\gamma_k'')}\left(\widehat{Q}_{\varkappa}[p]\right),
\end{equation}

\item[2)] $0=Z_{(\lambda_{k+1},\gamma_k'')}\left(\widehat{Q}_{\varkappa}[p]\right)$. Then analogously, Proposition~\ref{lemma.55} implies inequality~\eqref{Theorem.ineq.Q.Q.hat.greater.1/2.proof.8.1.2}.
\end{itemize}

\vspace{2mm}

\noindent Let now $n_{k+1}\geqslant3$, and $\dfrac12<\varkappa<\dfrac{n_{k+1}-1}{n_{k+1}}$, so that $0<2-\dfrac1{\varkappa}<\dfrac{n_{k+1}-2}{n_{k+1}-1}$. Then the zeroes of $Q_{\varkappa}[p]$ and~$Q_{\varkappa}[p']$ on the interval $(\gamma_k',\gamma_k'')$ coincide with the solutions of equations~\eqref{Theorem.ineq.Q.Q.hat.greater.1/2.proof.0}, respectively. By concavity of functions $R[p]$ and $R[p']$ and by~\eqref{R.at.multiple.zero}, we obtain that both equations~\eqref{Theorem.ineq.Q.Q.hat.greater.1/2.proof.0} have exactly two solutions (of multiplicity one) on $(\gamma_k',\gamma_k'')$, therefore,
\begin{equation}\label{Theorem.ineq.Q.Q.hat.greater.1/2.proof.8.1.3}
2=Z_{(\gamma_k',\gamma_k'')}(Q_{\varkappa}[p])\leqslant
Z_{(\gamma_{k}',\gamma_k'')}\left(\widehat{Q}_{\varkappa}[p]\right)=2.
\end{equation}

Finally, suppose that $n_{k+1}\geqslant3$, and $\varkappa=\dfrac{n_{k+1}-1}{n_{k+1}}$, so
$\varkappa=\dfrac{n_{k+1}-2}{n_{k+1}-1}$. In this case, from~\eqref{Real.poly.multiple.zero} and~\eqref{R.at.multiple.zero.expansion} we have
\begin{equation}\label{R.at.multiple.zero.expansion.2}
\begin{array}{l}
R(z)=\dfrac{n_{k+1}-1}{n_{k+1}}+\dfrac{2B}{An_{k+1}^2}(z-\lambda_{k+1})+\\
\\
+\dfrac{3\left[(n_{k+1}+1)B^2+2A\cdot C\cdot n_{k+1}\right]}{A^2n_{k+1}^3}(z-\lambda_{k+1})^2+O\left((z-\lambda_{k+1})^3\right)\qquad\text{as}\quad z\to\lambda_{k+1},
\end{array}
\end{equation}
and
\begin{equation}\label{R1.at.multiple.zero.expansion}
\begin{array}{l}
R[p'](z)=\dfrac{n_{k+1}-2}{n_{k+1}-1}+
\dfrac{2B(n_{k+1}+1)}{An_{k+1}(n_{k+1}-1)^2}(z-\lambda_{k+1})+\\
\\
+\dfrac{3\left[(n_{k+1}+1)^2B^2+2A\cdot C\cdot (n_{k+1}-1)(n_{k+1}+2)\right]}{A^2n_{k+1}(n_{k+1}-1)^3}(z-\lambda_{k+1})^2+O\left((z-\lambda_{k+1})^3\right),
\end{array}
\end{equation}
as $z\to\lambda_{k+1}$.

From~\eqref{R.at.multiple.zero.expansion.2} and~\eqref{R1.at.multiple.zero.expansion} it follows that the number $\lambda_{k+1}$ is a solution of equations~\eqref{Theorem.ineq.Q.Q.hat.greater.1/2.proof.0} of the same multiplicity. That is, it is simultaneously a simple or a multiple (of multiplicity $2$) solution of both equations on the interval $(\gamma_k',\gamma_{k+1}'')$ (even for $n_{k+1}=2$). Therefore, by Theorem~\ref{Theorem.equations.equivalence}, the functions $Q_{\varkappa}[p]$ and
$\widehat{Q}_{\varkappa}[p]$ have simultaneously one or no zeroes (counting multiplicities) on the interval $(\gamma_k',\gamma_{k+1}'')$. Consequently, we obtain
\begin{equation}\label{Theorem.ineq.Q.Q.hat.greater.1/2.proof.8.1.4}
Z_{(\gamma_k',\gamma_k'')}(Q_{\varkappa}[p])\leqslant
Z_{(\gamma_{k}',\gamma_k'')}\left(\widehat{Q}_{\varkappa}[p]\right),
\end{equation}
in this case.

Now from~\eqref{Theorem.ineq.Q.Q.hat.greater.1/2.proof.8.1}, \eqref{Theorem.ineq.Q.Q.hat.greater.1/2.proof.8.1.2}, \eqref{Theorem.ineq.Q.Q.hat.greater.1/2.proof.8.1.3}, and \eqref{Theorem.ineq.Q.Q.hat.greater.1/2.proof.8.1.4}, it follows that
inequality~\eqref{auxiliary.main.ineq.mids} holds for any $\dfrac12<\varkappa<\dfrac{n-1}{n}$ in the case when $\lambda_{k+1}$ is a multiple
zero of $p$ as well, as required.
\end{proof}

Now Theorem~\ref{Theorem.ineq.Q.Q.hat.greater.1/2} follows from Lemmas~\ref{Lemma.main.ends}--\ref{Lemma.main.mids}, since
both functions $Q_{\varkappa}[p]$ and $\widehat{Q}_{\varkappa}[p]$ do not
vanish at the points $\mu_k$, $k=1,\ldots,d-1$.

\vspace{2mm}

Now we are in a position to find the lower bound for a number of non-real zeroes of the polynomial~$F_{\varkappa}[p]$.

\begin{theorem}\label{Theorem.lower.bound.1}
	Let $p$ be a real polynomial with only real zeroes given in~\eqref{Lemma.concave.R.proof.1}, and let its zeroes be indexed in the following
	order
	\begin{equation*}
	\lambda_1<\lambda_2<\cdots<\lambda_d,
	\end{equation*}
	where $d$ is the number of distinct zeroes of $p$. Then the
	following inequalities hold:
	\vspace{1mm}
	\begin{itemize}
		\item[] if $\ \ \dfrac{m_1-1}{m_1}<\varkappa<\dfrac{n-d+1}{n-d+2}$, then $F_{\varkappa}[p]$ can have only real zeroes
		\begin{equation}\label{Theorem.lower.bound.1.estimates.1}
		Z_{C}(F_{\varkappa}[p])\geqslant0,
		\end{equation}
		\item[] if $\ \ \dfrac{n-d+k-1}{n-d+k}\leqslant\varkappa<\dfrac{n-d+k}{n-d+k+1}$, $k=2,\ldots,d-1$, then
		
		\vspace{1mm}

		\begin{equation}\label{Theorem.lower.bound.1.estimates.2}
		Z_{C}(F_{\varkappa}[p])\geqslant2k-2,
		\end{equation}
	\end{itemize}
	where $m_1$ is defined in~\eqref{order.multiplicities.1}.
\end{theorem}
\begin{proof}
We prove this theorem by induction. Suppose first that all zeroes of the polynomial $p$ are simple, that is, $d=n$ and $m_1=1$. Then inequalities~\eqref{Theorem.lower.bound.1.estimates.1}--\eqref{Theorem.lower.bound.1.estimates.2}
have the following form
	\begin{itemize}
	\item[] if $\ \ 0<\varkappa<\dfrac{1}{2}$, then
	\begin{equation}\label{Theorem.lower.bound.1.estimates.1.proof.1}
	Z_{C}(F_{\varkappa}[p])\geqslant0,
	\end{equation}
	\item[] if $\ \ \dfrac{k-1}{k}\leqslant\varkappa<\dfrac{k}{k+1}$, $k=2,\ldots,n-1$, then
	
	\vspace{1mm}

	\begin{equation}\label{Theorem.lower.bound.1.estimates.2.proof.1}
	Z_{C}(F_{\varkappa}[p])\geqslant2k-2,
	\end{equation}
\end{itemize}

Since all zeroes of $p$ are simple, the set of all non-trivial zeroes of $F_{\varkappa}[p]$ (the set of all zeroes of~$Q_{\varkappa}[p]$) coincides
with the set of all solutions of the equation $R[p](z)=\varkappa$ by Theorem~\ref{Theorem.equations.equivalence}. Moreover, if $0<\varkappa<\dfrac{1}{2}$, then
the equation $R[p](z)=\varkappa$ may have all real solutions. Indeed, if the polynomial $p$ is such that
$p$ and $p''$ have no common zeroes, then the function $R[p]$ has exactly two \textit{distinct} zeroes on every
interval $(\mu_{k},\mu_{k+1})$, $k=1,\ldots,n-2$, and exactly one (counting multiplicities) zero on each interval $(-\infty,\mu_1)$ and $(\mu_{n-1},+\infty)$ by Theorem~\ref{Theorem.R.zeroes}. Thus, on each interval $(\mu_{k},\mu_{k+1})$, $k=1,\ldots,n-2$, the function $R[p]$ has the maximum $M_k$. Consequently, for any $0<\varkappa<\min\{M_1,\ldots,M_{n-1}\}$ all solutions of the equation $R[p](z)=\varkappa$ are real. At the same time, for any $n\geqslant3$ each zero of the second derivative of the polynomial $q_n(z)=(z^2-1)P_{n-2}^{(1,1)}(z)$,
where $P_m^{(\alpha,\beta)}$ is the Jacobi polynomial, is a zero of the polynomial $q_n$ itself, see, e.g.~\cite{Andrews_Askey_Roy}. Therefore, in this case, the equation $R[p](z)=\varkappa$ has
exactly $2$ real solutions for any $0<\varkappa<\dfrac{n-1}{n}$. Thus, we obtain that
\begin{equation}\label{Theorem.Q.simple.real.zeros.proof.0}
2\leqslant Z_{\mathbb{R}}(Q_{\varkappa}[p])\leqslant2n-2,
\end{equation}
for any $0<\varkappa<\dfrac12$.

Let $\varkappa\in\left(\dfrac{k-1}{k},\dfrac{k}{k+1}\right)$ for some number $k$, $k=2,\ldots,n-1$. Introduce the following sequence of numbers
\begin{equation}\label{Theorem.Q.simple.real.zeros.proof.1}
\varkappa_{0}\eqdef \varkappa,\qquad \varkappa_{i}\eqdef 2-\dfrac1{\varkappa_{i-1}},\quad i=1,\ldots,k-1,
\end{equation}
and note that $\varkappa_{k-1}\in\left(0,\dfrac12\right)$. Since $\widehat{Q}_{\varkappa_{i}}[p](z)\equiv Q_{\varkappa_{i+1}}[p'](z)$ and $\deg p^{(k-1)}=n-k+1$,
we have from~\eqref{auxiliary.main.ineq} and~\eqref{Theorem.Q.simple.real.zeros.proof.0}
\begin{equation}\label{Theorem.Q.simple.real.zeros.proof.2}
Z_{\mathbb{R}}(Q_{\varkappa}[p])=Z_{\mathbb{R}}(Q_{\varkappa_0}[p])\leqslant Z_{\mathbb{R}}(Q_{\varkappa_1}[p'])\leqslant\ldots\leqslant Z_{\mathbb{R}}(Q_{\varkappa_{k-1}}[p^{(k-1)}])\leqslant2(n-k+1)-2=2n-2k
\end{equation}
for any $\dfrac{k-1}{k}<\varkappa<\dfrac{k}{k+1}$, $k=2,\ldots,n-1$.

Suppose now  that $\varkappa=\dfrac12$.
On each interval $(\mu_k,\mu_{k+1})$, $k=1,\ldots,n-2$, there lies a zero $\lambda_{k+1}$ of $p$
and a zero $\gamma_k$ of $p''$. By Theorem~\ref{Theorem.equations.equivalence}, the function $Q_{\tfrac12}[p]$ has at most $2$ zeroes on $(\mu_k,\mu_{k+1})$, counting multiplicities.

If $\lambda_{k+1}=\gamma_k$, then the function $Q_{\tfrac12}[p]$ has no zeroes on $(\mu_k,\mu_{k+1})$ due to concavity of $R[p]$.

Let $\lambda_{k+1}<\gamma_k$ (the case $\lambda_{k+1}>\gamma_k$ can be considered similarly). Then by concavity of $R[p]$, the function~$Q_{\tfrac12}[p]$ has no zeroes
on the intervals $(\mu_{k},\lambda_{k+1}]$ and $[\gamma_k,\mu_{k+1})$ and has an
even number (at~most~$2$) on the interval $(\lambda_{k+1},\gamma_k)$. The function $\widehat{Q}_{\tfrac12}[p]$ has at most one zero,
counting multiplicities, on $(\lambda_{k+1},\gamma_k)$, since it has at most two zeroes on the interval $[\mu_{k},\gamma_k)$
due to concavity of the function~$R[p']$.

If $\widehat{Q}_{\tfrac12}[p]$ does not vanish on $(\lambda_{k+1},\gamma_k)$, then by Proposition~\ref{lemma.55} one
has $Z_{(\lambda_{k+1},\gamma_k)}({Q}_{\varkappa})=0$. If there is a (simple) zero of $\widehat{Q}_{\tfrac12}[p]$
on $(\lambda_{k+1},\gamma_k)$, then $Z_{(\lambda_{k+1},\gamma_k)}({Q}_{\varkappa})=0$ or $2$ according to~\eqref{main.work.formula.11}.
Since $\widehat{Q}_{\tfrac12}[p](z)=\dfrac{p'''(z)}{p'(z)}$, there are at most $n-3$ intervals $(\mu_k,\mu_{k+1})$
where $Q_{\tfrac12}[p]$ has exactly two zeroes (counting multiplicities). Moreover, it has at most one zero
(counting multiplicities) on each interval $(-\infty,\mu_1)$ and $(\mu_{n-1},+\infty)$ by
Theorems~\ref{Theorem.R.zeroes}--\ref{Theorem.equations.equivalence}. Thus, we obtain that
\begin{equation}\label{Theorem.lower.bound.1.proof.1}
0\leqslant Z_{\mathbb{R}}(Q_{\tfrac12}[p])\leqslant 2n-4,
\end{equation}
since $\mu_k$, $k=1,\ldots,n-1$, are not zeroes of $Q_{\tfrac12}[p]$.

Let now $\varkappa=\dfrac{k-1}{k}$ for some $k=3,\ldots,n-1$. Then with the
numbers~\eqref{Theorem.Q.simple.real.zeros.proof.1}, we have
\begin{equation}\label{Theorem.Q.simple.real.zeros.proof.3}
\varkappa_0=\varkappa=\dfrac{k-1}{k},\qquad \varkappa_i=2-\dfrac1{\varkappa_{i-1}},
\qquad i=1,\ldots,k-2,
\end{equation}
where $\varkappa_{k-2}=\dfrac12$. Now from~\eqref{auxiliary.main.ineq} and~\eqref{Theorem.lower.bound.1.proof.1} it follows
\begin{equation}\label{Theorem.Q.simple.real.zeros.proof.22223}
Z_{\mathbb{R}}(Q_{\varkappa}[p])=Z_{\mathbb{R}}(Q_{\varkappa_0}[p])\leqslant Z_{\mathbb{R}}(Q_{\varkappa_1}[p'])\leqslant\ldots\leqslant Z_{\mathbb{R}}(Q_{\varkappa_{k-2}}[p^{(k-2)}])\leqslant2(n-k+2)-4=2n-2k,
\end{equation}
since $\widehat{Q}_{\varkappa_{i}}[p](z)\equiv Q_{\varkappa_{i+1}}[p'](z)$ and $\deg p^{(k-2)}=n-k+2$.
Thus, inequalities~\eqref{Theorem.Q.simple.real.zeros.proof.0},~\eqref{Theorem.Q.simple.real.zeros.proof.2}
and~\eqref{Theorem.Q.simple.real.zeros.proof.22223} hold for
polynomials with only simple zeroes.

For the sequel, it is more convenient to rewrite~\eqref{Theorem.Q.simple.real.zeros.proof.2}
and~\eqref{Theorem.Q.simple.real.zeros.proof.22223} as follows.
\begin{equation}\label{Theorem.Q.simple.real.zeros.proof.2223}
Z_{\mathbb{R}}(Q_{\varkappa}[p])\leqslant2j,
\end{equation}
for $\dfrac{n-j-1}{n-j}\leqslant\varkappa<\dfrac{n-j}{n-j+1}$, $j=1,\ldots,n-2$.

\vspace{2mm}

Suppose now that the polynomial has roots of multiplicity at most $2$. Then its derivative has only simple zeroes. Then by Theorem~\ref{Theorem.ineq.Q.Q.hat.greater.1/2},~\eqref{Theorem.Q.simple.real.zeros.proof.0}, and~\eqref{Theorem.Q.simple.real.zeros.proof.2223}, one has
\begin{equation}\label{Theorem.Q.simple.real.zeros.proof.0_1}
2\leqslant Z_{\mathbb{R}}(Q_{\varkappa}[p])\leqslant Z_{\mathbb{R}}\left(\widehat{Q}_{\varkappa}[p]\right)\leqslant2d-2,
\end{equation}
for any $\dfrac{n-d}{n-d+1}<\varkappa<\dfrac{n-d+1}{n-d+2}$, and
\begin{equation}\label{Theorem.Q.simple.real.zeros.proof.2223_1}
Z_{\mathbb{R}}(Q_{\varkappa}[p])\leqslant Z_{\mathbb{R}}\left(\widehat{Q}_{\varkappa}[p]\right)\leqslant2j,
\end{equation}
for $\dfrac{n-j-1}{n-j}\leqslant\varkappa<\dfrac{n-j}{n-j+1}$, $j=1,\ldots,d-2$.

\vspace{2mm}

It is clear now that if inequalities~\eqref{Theorem.Q.simple.real.zeros.proof.0_1}--\eqref{Theorem.Q.simple.real.zeros.proof.2223_1} are true for polynomials
whose zeroes have multiplicity at most $M$, $1\leqslant M\leqslant n-2$, then
they hold for polynomials whose zeroes have multiplicity at most $M+1$ according to Theorem~\ref{Theorem.ineq.Q.Q.hat.greater.1/2}, since for $p'$ they are true by assumption. Consequently,
inequalities~\eqref{Theorem.Q.simple.real.zeros.proof.0_1}--\eqref{Theorem.Q.simple.real.zeros.proof.2223_1} hold for any polynomial with only real zeroes.

Note now that for $\dfrac{m_1-1}{m_1}<\varkappa\leqslant\dfrac{n-d}{n-d+1}$, where $m_1$ is defined in~\eqref{order.multiplicities.1}, we can conclude that $Q_{\varkappa}[p]$
may have only real zeroes as well as in~\eqref{Theorem.Q.simple.real.zeros.proof.0_1}. Therefore,
\begin{equation}\label{Theorem.Q.simple.real.zeros.proof.0_11}
2\leqslant Z_{\mathbb{R}}(Q_{\varkappa}[p])\leqslant Z_{\mathbb{R}}\left(\widehat{Q}_{\varkappa}[p]\right)\leqslant2d-2,
\end{equation}
for any $\dfrac{m_1-1}{m_1}<\varkappa\leqslant\dfrac{n-d}{n-d+1}$.

\vspace{2mm}

Let us now denote by $N_{\mathbb{R}}^{(\varkappa)}$ and $\widehat{N}_{\mathbb{R}}^{(\varkappa)}$ the number of real solutions of equations~\eqref{Theorem.ineq.Q.Q.hat.greater.1/2.proof.0}, respectively. If $\varkappa\neq\dfrac{k-1}{k}$, $k=1,\ldots,n$, the inequalities~\eqref{Theorem.Q.simple.real.zeros.proof.0_1}--\eqref{Theorem.Q.simple.real.zeros.proof.2223_1}
hold for $N^{(\varkappa)}$ by Theorem~\ref{Theorem.equations.equivalence}.

Furthermore, from~\eqref{R.at.multiple.zero.expansion.2}--\eqref{R1.at.multiple.zero.expansion} it follows that if a zero $\lambda_k$ of the polynomial $p$ is multiple, then it is a solution to
both equations~\eqref{Theorem.ineq.Q.Q.hat.greater.1/2.proof.0} of the same multiplicity ($1$ or $2$). Since all other solutions to these equations (and only they) are zeroes of the functions
$Q_{\varkappa}[p]$ and $\widehat{Q}_{\varkappa}[p]$ for $\dfrac12<\varkappa<\dfrac{n-1}{n}$, we obtain from Theorem~\ref{Theorem.ineq.Q.Q.hat.greater.1/2} the following inequality
\begin{equation}\label{Theorem.lower.bound.1.proof.2}
N_{\mathbb{R}}^{(\varkappa)}\leqslant\widehat{N}_{\mathbb{R}}^{(\varkappa)}
\end{equation}
for $\dfrac12<\varkappa<\dfrac{n-1}{n}$.

Now if $p$ has only simple zeroes, then $N_{\mathbb{R}}^{\left(\tfrac12\right)}=Z_{\mathbb{R}}\left(Q_{\tfrac12}[p]\right)$,
so we have
\begin{equation}\label{Theorem.Q.simple.real.zeros.proof.2223_22}
N_{\mathbb{R}}^{\left(\tfrac12\right)}\leqslant2n-4
\end{equation}
by~\eqref{Theorem.lower.bound.1.proof.1}. Now using~\eqref{Theorem.lower.bound.1.proof.2}--\eqref{Theorem.Q.simple.real.zeros.proof.2223_22}, one can prove
that inequalities~\eqref{Theorem.Q.simple.real.zeros.proof.0_1}--\eqref{Theorem.Q.simple.real.zeros.proof.2223_1} hold for $N_{\mathbb{R}}^{(\varkappa)}$
whenever $\dfrac{m_1-1}{m_1}<\varkappa<\dfrac{n-1}{n}$. The proof is similar to the one we used to 
prove~\eqref{Theorem.Q.simple.real.zeros.proof.0_1}--\eqref{Theorem.Q.simple.real.zeros.proof.2223_1}.

Consequently, for $N_{\mathbb{R}}^{(\varkappa)}$ the following inequalities hold
\begin{equation}\label{Theorem.Q.simple.real.zeros.proof.2223_2}
N_{\mathbb{R}}^{(\varkappa)}\leqslant2j,
\end{equation}
for $\dfrac{n-j-1}{n-j}\leqslant\varkappa\leqslant\dfrac{n-j}{n-j+1}$, $j=1,\ldots,d-2$ and
\begin{equation}\label{Theorem.Q.simple.real.zeros.proof.2223_212}
N_{\mathbb{R}}^{(\varkappa)}\leqslant2d-2,
\end{equation}
for $\dfrac{m_1-1}{m_1}\leqslant\varkappa\leqslant\dfrac{n-d+1}{n-d+2}$.

\vspace{2mm}

Recall that by Theorem~\ref{Theorem.equations.equivalence} the number of solutions of the equation $R[p](z)=\varkappa$ equals $2d-2$ where~$d$ is the number of distinct zeroes of $p$, and the set of non-real zeroes of the polynomial $F_{\varkappa}[p]$ coincides with the set of non-real solutions of the equation $R[p](z)=\varkappa$. Now inequalities~\eqref{Theorem.lower.bound.1.estimates.1}--\eqref{Theorem.lower.bound.1.estimates.2} follow from~\eqref{Theorem.Q.simple.real.zeros.proof.0_11}--\eqref{Theorem.Q.simple.real.zeros.proof.2223_212}, as required.
\end{proof}

\setcounter{equation}{0}
\section{Polynomials with non-real zeroes}\label{section:non-real.zeroes}

In this section, we disprove a conjecture by B.\,Shapiro~\cite{Shapiro} and discuss possible extensions of our
results from Section~\ref{section:real.zeroes} for arbitrary real polynomials.

Let $p$ be an arbitrary real polynomial. In this case, the polynomial $F_{\varkappa}[p]$
can have both real and non-real trivial zeroes. So, to study non-trivial zeroes of
$F_{\varkappa}[p]$ it is more convenient to consider the function~$Q_{\varkappa}[p]$
defined in~\eqref{main.function.Q}.

In~\cite[Conjecture 11]{Shapiro}, the following analogue of the Hawaii conjecture~\cite{CCS_1987,Tyaglov_Hawaii} appeared.
\begin{conjecture}[B.\,Shapiro]\label{Conjecture.Shapiro}
Let $p$ be an arbitrary real polynomial of degree $n$, $n\geqslant2$, then
\begin{equation}\label{Conjecture.Shapiro.ineq}
Z_{\mathbb{R}}\left(Q_{\tfrac{n-1}{n}}[p]\right)\leqslant Z_{C}(p).
\end{equation}
\end{conjecture}
The Hawaii conjecture posed in~\cite{CCS_1987} and proved in~\cite{Tyaglov_Hawaii} states that inequality~\eqref{Conjecture.Shapiro.ineq} is true for the function~$Q_1[p]$.
As it was shown in Sections~\ref{section:function.Q}--\ref{section:real.zeroes}, the value $\varkappa=\dfrac{n-1}{n}$ is important, and for polynomials with real zeroes
the properties of $Q_{\tfrac{n-1}{n}}[p]$ are close to the ones of $Q_{1}[p]$. However, Conjecture~\ref{Conjecture.Shapiro} is not true.

Indeed, consider the polynomial
\begin{equation}\label{counter1}
p(z)=(z^2+a^2)(z+a^2)(z-1),\qquad a\in\mathbb{R}\backslash\{-1,0,1\},
\end{equation}
of degree $4$. It has two distinct real zeroes, $-a^2$ and $1$, and two non-real zeroes $\pm ia$, so $Z_{C}(p)=2$. For this polynomial,
the function~$Q_{\tfrac{n-1}{n}}[p]$ has the form
\begin{equation*}
Q_{\tfrac34}[p](z)=-\dfrac34\cdot\dfrac{(a^2-1)^2z^4-8a^2(a^2-1)z^3-2a^2(a^4-10a^2+1)z^2+8a^4(a^2-1)z+a^4(a^2-1)^2}{(z^2+a^2)^2(z+a^2)^2(z-1)^2}.
\end{equation*}
This rational function has four zeroes
\begin{equation*}
\lambda_1=\lambda_2=\dfrac{a(a+1)}{a-1},\quad\lambda_3=\lambda_4=-\dfrac{a(a-1)}{a+1},
\end{equation*}
all of which are real whenever $a\in\mathbb{R}\backslash\{-1,1\}$, so
\begin{equation}\label{counter1.ineq}
Z_{\mathbb{R}}\left(Q_{\tfrac{3}{4}}[p]\right)=4>2=Z_{\mathbb{C}}(p),
\end{equation}
and Conjecture~\ref{Conjecture.Shapiro} fails. 

\begin{remark}
Conjecture 11 in~\cite{Shapiro}, in fact, looks as follows
\begin{equation*}
Z_{C}\left(F_{\tfrac{n-1}{n}}[p]\right)\leqslant Z_{C}(p)
\end{equation*}
with additional condition that all real zeroes of $p$ are simple. It is easy to see that this conjecture is equivalent to Conjecture~\ref{Conjecture.Shapiro}
in the considered case, since $Z_{\mathbb{R}}\left(F_{\tfrac{n-1}{n}}[p]\right)=Z_{\mathbb{R}}\left(F_{\tfrac{n-1}{n}}[p]\right)$ whenever real zeroes of~$p$ are simple.
\end{remark}

Let us look at inequality~\eqref{counter1.ineq} from the point of view of the function $R$ defined in~\eqref{function.R}. For the polynomial~\eqref{counter1},
the rational function $R$ has the form
\begin{equation*}
R(z)=\dfrac{p(z)p''(z)}{[p'(z)]^2}=\dfrac{6z(2z-1+a^2)(z^2+a^2)(z+a^2)(z-1)}{(4z^3-3z^2+3z^2a^2-a^2+a^4)^2}.
\end{equation*}
From~\eqref{counter1.ineq} it follows that the equation
\begin{equation*}
R(z)=\dfrac{3}{4}
\end{equation*}
has $4$ real solutions. The function~$R$ is drawn at Fig.~\ref{pic.2} for $a=10$.
\begin{figure}[ht]
\centering \includegraphics[scale=0.28]{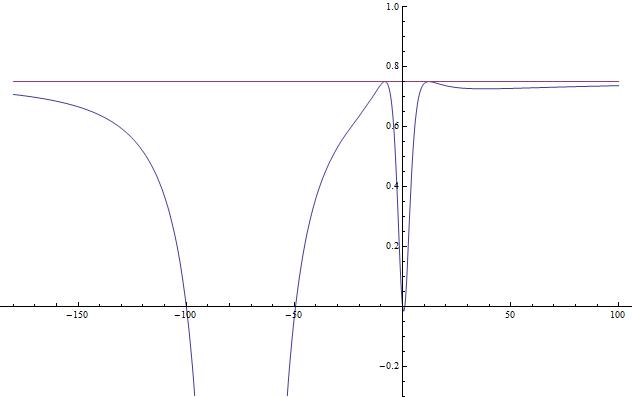} \caption{Function $R$ for the polynomial~\eqref{counter1}.}
\label{pic.2}
\end{figure}
It is easy to see that $R$ has two maxima with maximum values $\dfrac34$. In fact, simple calculations show that $R$ has two
maximum points at $-\dfrac{a(a-1)}{a+1}$ and $\dfrac{a(a+1)}{a-1}$ with the value $\dfrac34$.

Thus, according to this counterexample we conjecture another inequality generalizing the Hawaii conjecture.
\begin{conjecture}\label{Conjecture.new}
Let $p$ be  a real polynomial of degree $n$, $n\geqslant2$. Then
\begin{equation*}
Z_{C}(p)-Z_{C}(p')\leqslant Z_{R}(Q_{\varkappa}[p])\leqslant Z_{C}(p)
\end{equation*}
for $\varkappa>\dfrac{n-1}{n}$.
\end{conjecture}
\noindent This conjecture is proved only for the case $\varkappa=1$, see~\cite{Tyaglov_Hawaii}.

For $\varkappa\leqslant\dfrac{n-1}{n}$, it is not easy to predict the estimates for the number of real zeroes of $Q_{\varkappa}[p]$. However, calculations
show that the following conjecture for a special case of polynomials may be true.
\begin{conjecture}\label{Conjecture.new.2}
Let $p$ be a real polynomial of degree $n=2m$ with only non-real zeroes, and let $p'$ have only real zeroes. Then
\begin{equation*}
Z_{R}(Q_{\varkappa}[p])=Z_{C}(p)\qquad\text{for}\quad\varkappa>\dfrac{n-1}{n},
\end{equation*}
\begin{equation*}
Z_{C}(p)-2\leqslant Z_{R}(Q_{\varkappa}[p])\leqslant Z_{C}(p)\qquad\text{for}\quad\dfrac12<\varkappa\leqslant\dfrac{n-1}{n},
\end{equation*}
and
\begin{equation*}
Z_{R}(Q_{\varkappa}[p])=Z_{C}(p)-2\qquad\text{for}\quad\varkappa\leqslant\dfrac{1}{2}.
\end{equation*}
\end{conjecture}
\noindent In~\cite{CCS_1987,Tyaglov_Hawaii} it was established that this conjecture is true for $\varkappa=1$.

Note that the polynomial $p$ defined in~\eqref{counter1} does not satisfy the conditions of
Conjecture~\ref{Conjecture.new.2} since its derivative always has non-real zeroes whenever $a\neq\pm1,0$.
\setcounter{equation}{0}
\section{Conclusion and open problems}\label{section:open.problems}

In the present work, we find the upper and lower bounds for the number of non-real zeroes of the
differential polynomial
\begin{equation*}
F_{\varkappa}[p](z)=p(z)p''(z)-\varkappa[p'(z)]^2,
\end{equation*}
whenever $p$ has only real zeroes. We also disprove a conjecture by B.\,Shapiro~\cite{Shapiro} on the number
of real zeroes of $F_{\varkappa}[p]$ for arbitrary real polynomial $p$. Instead, we provide two
new conjectures that generalise the Hawaii conjecture~\cite{CCS_1987} proved in~\cite{Tyaglov_Hawaii}.
We believe that our method of combining Proposition~\ref{lemma.55} and some properties of the
function $R$ defined in~\eqref{function.R} can be useful for proof of these conjectures
and might provide a new, more simple, proof to the Hawaii conjecture.

Finally, we note that it is intuitively clear that our results of Section~\ref{section:real.zeroes} can be
extended to the case when $p$ is an entire function in a subclass of the Laguerre-P\'olya class $\mathcal{L-P}$ 
(or even if $p$ belongs to a subclass of $U_{2p}$, see, e.g.,~\cite{EdwardsHellerstein} and references there).
For example, it must be true for entire functions whose supremum of multiplicities of their zeroes is finite,
in particular, if they have finitely many zeroes or only simple zeroes.

\section*{Acknowledgement}
The work of M.\,Tyaglov was partially supported by The Program for Professor of Special Appointment (Oriental Scholar) at Shanghai Institutions of Higher Learning and by National Natural Science Foundation
of China under Grant No. 11871336.

M.\,Tyaglov thanks Anna Vishnyakova and Alexander Dyachenko for helpful discussions.

M.J.~Atia acknowledges the hospitality of the School of Mathematical Sciences of Shanghai Jiao Tong University in June 2015 and June 2016 partially supported by The Program for Professor of Special Appointment (Oriental Scholar) at Shanghai Institutions of Higher Learning of M. Tyaglov.

\appendix

\setcounter{equation}{0}

\section{Appendix: Proof of Proposition~\ref{lemma.55}}\label{section:appendix.finite.intervals}

Let~$p$ be a real polynomial of degree $n$, $n\geqslant2$. Recall that by $F_{\varkappa}[p]$ we
denote the following polynomial
\begin{equation}\label{F.function.1}
F_{\varkappa}[p](z)=p(z)p''(z)-\varkappa\left[p'(z)\right]^2,
\end{equation}
defined in~\eqref{main.polynomial.F}. By $\widehat{F}_{\varkappa}[p]$ we denote the polynomial
\begin{equation}\label{F.function.2}
\widehat{F}_{\varkappa}[p](z)=F_{2-\tfrac1{\varkappa}}[p'](z)=p'(z)p'''(z)-\left(2-\dfrac1{\varkappa}\right)\left[p''(z)\right]^2.
\end{equation}
The functions $Q_{\varkappa}[p]$ and $\widehat{Q}_{\varkappa}[p]$ are defined in~\eqref{main.function.Q}
and~\eqref{Q.hat}, respectively.

Our first result is about the number of zeroes of $Q_{\varkappa}[p]$ on a finite
interval free of zeroes of the functions $Q_{\varkappa}[p]$, $p'$,
$p''$ and $\widehat{Q}_{\varkappa}[p]$.

\begin{lemma}\label{lemma.3}
Let $p$ be a real polynomial, $\varkappa>0$, $a,b\in\mathbb{R}$, and let
$p(z)\neq0$, $p'(z)\neq0$, $p''(z)\neq0$,
$\widehat{Q}_{\varkappa}[p](z)\neq0$ in the~interval~$(a,b)$. Suppose additionally that
if $p(b)\neq0$ then $p'(b)\neq0$ as well.
\begin{itemize}
\item [\emph{I.}] If, for all sufficiently small $\delta>0$,
\begin{equation}\label{lemma.3.condition.1}
p'(a+\delta)p''(a+\delta)Q_{\varkappa}[p](a+\delta)\widehat{Q}_{\varkappa}[p](a+\delta)>0,
\end{equation}
then $Q_{\varkappa}[p]$ has no zeroes in $(a,b]$.
\item [\emph{II.}] If, for all sufficiently small $\delta>0$,
\begin{equation}\label{lemma.3.condition.2}
p'(a+\delta)p''(a+\delta)Q_{\varkappa}[p](a+\delta)\widehat{Q}_{\varkappa}[p](a+\delta)<0,
\end{equation}
then $Q_{\varkappa}[p]$ has at most one zero in $(a,b)$, counting multiplicities.
Moreover, if~$Q_{\varkappa}[p](\zeta)=0$ for some $\zeta\in(a,b)$, then
$Q_{\varkappa}[p](b)\neq0$ \emph{(}if $Q_{\varkappa}[p]$ is finite at $b$\emph{)}.
\end{itemize}
\end{lemma}
\begin{proof} The condition $p(z)\neq0$ for $z\in(a,b)$ means
that $Q_{\varkappa}[p](z)$ is finite at every point of~$(a,b)$. Note that
from~\eqref{F.function.1} it follows that
\begin{equation*}
p'(z)=\dfrac{p(z)p''(z)}{\varkappa p'(z)}-\dfrac{F_{\varkappa}[p](z)}{\varkappa p'(z)},
\end{equation*}
since $p'(z)\neq0$ for $z\in(a,b)$ by assumption. Substituting this expression into  formula~\eqref{F.function.2},
we obtain
\begin{equation}\label{new}
\begin{array}{c}
\widehat{F}_{\varkappa}[p](z)=\dfrac{p(z)p''(z)p'''(z)}{\varkappa p'(z)}-
(2\varkappa-1)\cdot\dfrac{p'(z)p''(z)p''(z)}{\varkappa p'(z)}-\dfrac{F_{\varkappa}[p](z)p'''(z)}{\varkappa p'(z)}=\\
\\
=\dfrac{p''(z)}{\varkappa p'(z)}\cdot
F'_{\varkappa}[p](z)-\dfrac{F_{\varkappa}[p](z)p'''(z)}{\varkappa p'(z)}.
\end{array}
\end{equation}

If $\zeta\in(a,b)$ and $Q_{\varkappa}[p](\zeta)=0$, then $F_{\varkappa}[p](\zeta)=0$
and~\eqref{new} implies
\begin{equation}\label{main.work.formula}
\widehat{F}_{\varkappa}[p](\zeta)=\dfrac{p''(\zeta)}{\varkappa p'(\zeta)}F'_{\varkappa}[p](\zeta).
\end{equation}
Since $p'(z)\neq0,p''(z)\neq0,\widehat{Q}_{\varkappa}[p](z)\neq0$ (and
therefore $\widehat{F}_{\varkappa}[p](z)\neq0$) in $(a,b)$ by assumption,
from~\eqref{main.work.formula} it follows that $\zeta$ is a~simple
zero of $Q_{\varkappa}[p]$. That is, all zeroes of $Q_{\varkappa}[p]$ in~$(a,b)$ are simple.


\begin{itemize}
\item[I.] Let inequality~\eqref{lemma.3.condition.1} hold.
Assume that, for all sufficiently small~$\delta>0$,
\begin{equation}\label{lemma.3.proof.1}
p'(a+\delta)p''(a+\delta)\widehat{Q}_{\varkappa}[p](a+\delta)>0,
\end{equation}
then $Q_{\varkappa}[p](a+\delta)>0$, that is, $F_{\varkappa}[p](a+\delta)>0$. Therefore, if
$\zeta$ is the leftmost zero of $Q_{\varkappa}[p]$ in~$(a,b)$, then
$F_{\varkappa}'[p](\zeta)<0$. This inequality
contradicts~\eqref{main.work.formula}, since
\begin{equation*}
p'(z)p''(z)\widehat{Q}_{\varkappa}[p](z)>0
\end{equation*}
for $z\in(a,b)$, which follows from~\eqref{lemma.3.proof.1} and
from the assumption of the lemma. Consequently, the function $Q_{\varkappa}[p]$~cannot have
zeroes in the interval~$(a,b)$ if the
inequalities~\eqref{lemma.3.condition.1}
and~\eqref{lemma.3.proof.1} hold. In the same way, one can prove
that if $ p'(a+\delta)p''(a+\delta)\widehat{Q}_{\varkappa}[p](a+\delta)<0$
for all sufficiently small~$\delta>0$ and if the
inequality~\eqref{lemma.3.condition.1} holds, then $Q_{\varkappa}[p](z)\neq0$
for~$z\in(a,b)$.

Thus, $Q_{\varkappa}[p]$ has no zeroes in the interval~$(a,b)$ if the
inequality~\eqref{lemma.3.condition.1} holds. Moreover, it is easy
to show that $Q_{\varkappa}[p](b)\neq0$ as well. Indeed, let, on the contrary,
$Q_{\varkappa}[p](b)=0$. Then $p(b)\neq0$, therefore,
$p'(b)\neq0$ by assumption. So we obtain $F_{\varkappa}[p](b)=0$ and, from~\eqref{F.function.1},
$p''(b)\neq0$. Thus, we have $(pp'p'')(b)\neq0$. From~\eqref{new}--\eqref{main.work.formula} it follows
that the functions $F'_\varkappa[p]$ and $\widehat{F}_{\varkappa}[p]$ have a zero of the same order at~$b$. In particular, $F'_{\varkappa}[p](b)\neq0$
if and only if $\widehat{F}_{\varkappa}[p](b)\neq0$. Furthermore, it is clear that the order of the zero of $F'_{\varkappa}[p]$
(and $\widehat{F}_{\varkappa}[p]$) at $b$ is strictly smaller than the order of the zero of the polynomial $F_{\varkappa}[p]\cdot p'''$ at $b$. Consequently,
from~\eqref{new}--\eqref{main.work.formula} we obtain, for all sufficiently small~$\varepsilon>0$,
\begin{equation}\label{additional.work.formula.1}
\sgn\left(\dfrac{p'(b-\varepsilon)}
{p''(b-\varepsilon)}\widehat{F}_{\varkappa}[p](b-\varepsilon)\right)=\sgn(F'_{\varkappa}[p](b-\varepsilon)).
\end{equation}
But if the
inequality~\eqref{lemma.3.condition.1} holds, then
\begin{equation}\label{lemma.3.proof.2}
\sgn\left(\dfrac{p'(b-\varepsilon)}
{p''(b-\varepsilon)}\widehat{F}_{\varkappa}[p](b-\varepsilon)\right)=\sgn(F_{\varkappa}[p](b-\varepsilon))
\end{equation}
for all sufficiently small $\varepsilon>0$, since
$p'(z)\neq0$, $p''(z)\neq0$, $\widehat{Q}_{\varkappa}[p](z)\neq0$ in the
interval~$(a,b)$ by assumption and since $Q_{\varkappa}[p](z)\neq0$ in~$(a,b)$,
which was proved above. So, if the
inequality~\eqref{lemma.3.condition.1} holds and if $F_{\varkappa}[p](b)=0$, then
from~\eqref{additional.work.formula.1} and~\eqref{lemma.3.proof.2}
we obtain that
\begin{equation*}
F_{\varkappa}[p](b-\varepsilon)F'_{\varkappa}[p](b-\varepsilon)>0
\end{equation*}
for all sufficiently small $\varepsilon>0$. This inequality
contradicts the analyticity\footnote{If a real function $f$ is
analytic at some neighbourhood of a real point $a$ and equals zero
at this point, then, for all sufficiently small $\varepsilon>0$,
\begin{equation*}
f(a-\varepsilon)f'(a-\varepsilon)<0.
\end{equation*}
} of the polynomial~$F_{\varkappa}[p]$. Since $Q_{\varkappa}[p](b)$ exists, $\widehat{Q}_{\varkappa}[p](b)$ is finite by
assumption  of the lemma, so everything established above is true
for the functions~$Q_{\varkappa}[p]$ and $\widehat{Q}_{\varkappa}[p]$. Therefore, if the
inequality~\eqref{lemma.3.condition.1} holds and if $Q_{\varkappa}[p]$ is finite
at~the~point~$b$, then $Q_{\varkappa}[p](b)\neq0$. Thus, the first part of the
lemma is proved.
%
%
\item[II.] Let inequality~\eqref{lemma.3.condition.2}
hold, then $Q_{\varkappa}[p]$ can have zeroes in $(a,b)$. But it cannot have
more than one zero, counting multiplicity. In fact, if $\zeta$ is
the leftmost zero of $Q_{\varkappa}[p]$ in $(a,b)$, then this zero is simple as
we proved above. Therefore, the following inequality holds for all
sufficiently small $\varepsilon>0$
\begin{equation*}
p'(\zeta+\varepsilon)p''(\zeta+\varepsilon)Q_{\varkappa}[p](\zeta+\varepsilon)\widehat{Q}_{\varkappa}[p](\zeta+\varepsilon)>0.
\end{equation*}
Consequently, $Q_{\varkappa}[p]$ has no zeroes in $(\zeta,b]$ according to Case~I
of the lemma.
\end{itemize}
\end{proof}

Thus, we have found out that $Q_{\varkappa}[p]$ has at most one real zero,
counting multiplicity, in an interval where the functions
$p$, $p'$, $p''$ and $\widehat{Q}_{\varkappa}[p]$ have no real zeroes.
Now we study multiple zeroes of $Q_{\varkappa}[p]$ and its zeroes common with one
of the above-mentioned functions. From~\eqref{main.function.Q} it
follows that all zeroes of $p'$ of multiplicity at least $2$
that are not zeroes of $p$ are also zeroes of $Q_{\varkappa}[p]$ and all
zeroes of $p'$ of multiplicity at least $3$ that are not
zeroes of $p$ are multiple zeroes of $Q_{\varkappa}[p]$. The following
lemma provides information about common zeroes of $Q_{\varkappa}[p]$ and $\widehat{Q}_{\varkappa}[p]$.

\begin{lemma}\label{lemma.4}
Let $p$ be a real polynomial, $\varkappa>0$, $a,b\in\mathbb{R}$ be real and let
$p(z)\neq0$, $p'(z)\neq0$, $p''(z)\neq0$ in the
interval~$(a,b)$. Suppose that $\widehat{Q}_{\varkappa}[p]$ has a unique zero
$\xi\in(a,b)$ of multiplicity  $M$ in~$(a,b)$, and suppose
additionally that $p'(b)\neq0$ if $p(b)\neq0$.

If~$Q_{\varkappa}[p](\xi)=0$, then $\xi$ is a zero of $Q_{\varkappa}[p]$ of multiplicity $M+1$,
and $Q_{\varkappa}[p](z)\neq0$ for $z\in(a,\xi)\cup(\xi,b]$.
\end{lemma}
\begin{proof} The condition $p(z)\neq0$ for $z\in(a,b)$ means
that $Q_{\varkappa}[p]$ is finite at every point of~$(a,b)$.

By assumption, $\xi$ is a zero of $\widehat{F}_{\varkappa}[p]$ of multiplicity $M$ and
$F_{\varkappa}[p](\xi)=0$. First, we prove that $\xi$ is a~zero of~$F_{\varkappa}[p]$
of~multiplicity~$M+1$.

Note that the expression~\eqref{new} can be rewritten in the form
\begin{equation*}
\varkappa\dfrac{p'(z)}{\left[p''(z)\right]^2}\widehat{F}_{\varkappa}[p](z)=\left(\dfrac{F_{\varkappa}[p](z)}{p''(z)}\right)',
\end{equation*}
since $p'(z)\neq0$ and $p''(z)\neq0$ for $z\in(a,b)$ by assumption.
Differentiating this equality $j$ times with respect~to~$z$, we
get
\begin{equation}\label{derivatives}
\varkappa\left(\dfrac{p'(z)}{\left[p''(z)\right]^2}~\widehat{F}_{\varkappa}[p](z)\right)^{(j)}=\left(\dfrac{F_{\varkappa}[p](z)}{p''(z)}\right)^{(j+1)}.
\end{equation}
From~\eqref{derivatives} it follows that  $F_{\varkappa}^{(j+1)}[p](\xi)=0$ if
$p'(\xi)\neq0$, $p''(\xi)\neq0$, $\widehat{F}^{(i)}_{\varkappa}[p](\xi)=0$
and $F_{\varkappa}^{(i)}[p](\xi)=0$, $i=0,1,\ldots,j$. Consequently, $\xi$ is a
zero of $F_{\varkappa}[p]$ of multiplicity at least $M+1$. But by assumptions,
\eqref{derivatives} implies the following formula
\begin{equation*}
0\neq p'(\xi)\widehat{F}_{\varkappa}^{(M)}[p](\xi)=p''(\xi)F_{\varkappa}^{(M+1)}[p](\xi).
\end{equation*}
Hence, $\xi$ is a zero of $F_{\varkappa}[p]$ of multiplicity exactly $M+1$. But
$p(\xi)\neq0$ by assumption, therefore,~$\xi$~is a~zero of
$Q_{\varkappa}[p]$ of multiplicity $M+1$.

It remains to prove that $Q_{\varkappa}[p]$ has no zeroes in $(a,b]$ except
$\xi$. In fact, consider the interval~$(a,\xi)$. According to
Lemma~\ref{lemma.3}, $Q_{\varkappa}[p]$ can have a zero at $\xi$ only if the
inequality~\eqref{lemma.3.condition.2} holds and $Q_{\varkappa}[p](z)\neq0$ for
$z\in(a,\xi)$. Furthermore, the polynomial $p'\cdot p''$ does
not change its sign at $\xi$ but the function~$Q_{\varkappa}[p]\widehat{Q}_{\varkappa}[p]$ does, since
$\xi$ is a zero of~$Q_{\varkappa}[p]\cdot\widehat{Q}_{\varkappa}[p]$ of multiplicity~$2M+1$. Thus, for all
sufficiently small $\delta>0$,
\begin{equation}\label{additional.work.formula.2}
p'(\xi+\delta)p''(\xi+\delta)Q_{\varkappa}[p](\xi+\delta)\widehat{Q}_{\varkappa}[p](\xi+\delta)>0,
\end{equation}
since inequality~\eqref{lemma.3.condition.2} must hold in the
interval $(a,\xi)$ by Lemma~\ref{lemma.3}.
From~\eqref{additional.work.formula.2} it follows that Case~I of
Lemma~\ref{lemma.3} holds in the interval~$(\xi,b)$, so
$Q_{\varkappa}[p](z)\neq0$ for $z\in(\xi,b]$.
\end{proof}

Now combining the two last lemmas, we provide a general bound on
the number of real zeroes of~$Q_{\varkappa}[p]$ in terms of the number of real
zeroes of $\widehat{Q}_{\varkappa}[p]$ in a given interval.
\begin{theorem}\label{lemma.5}
Let $p$ be a real polynomial, $\varkappa>0$, and $a,b\in\mathbb{R}$.
If
$p(z)\neq0$, $p'(z)\neq0$ and $p''(z)\neq0$ for
$z\in(a,b)$, then
\begin{equation}\label{main.work.formula.1}
Z_{(a,b)}(Q_{\varkappa}[p])\leqslant1+Z_{(a,b)}\left(\widehat{Q}_{\varkappa}[p]\right).
\end{equation}
\end{theorem}
\begin{proof}
If $p(z)p''(z)<0$ in $(a,b)$, then $Q_{\varkappa}[p](z)<0$
for~$z\in(a,b)$ by~\eqref{main.function.Q}, that
is,~$Z_{(a,b)}(Q_{\varkappa}[p])=0$. Therefore, the
inequality~\eqref{main.work.formula.1} holds automatically in this
case.

Let now  $p(z)p''(z)>0$ for $z\in(a,b)$.
If~$\widehat{Q}_{\varkappa}[p](z)\neq0$ in~$(a,b)$, that is, $Z_{(a,b)}\left(\widehat{Q}_{\varkappa}[p]\right)=0$, then
by Lemma~\ref{lemma.3}, $Q_{\varkappa}[p]$ has at most one real zero, counting
multiplicity, in~$(a,b)$. Therefore,~\eqref{main.work.formula.1}
holds in this case.

If $\widehat{Q}_{\varkappa}[p]$ has a unique zero $\xi$ in~$(a,b)$ and $Q_{\varkappa}[p](\xi)\neq0$,
then by Lemma~\ref{lemma.3}, $Q_{\varkappa}[p]$ has at most one real zero,
counting multiplicity, in each interval $(a,\xi)$ and $(\xi,b)$:
\begin{equation}\label{lemma.5.proof.1}
Z_{(a,\xi)}(Q_{\varkappa}[p])\leqslant1+Z_{(a,\xi)}\left(\widehat{Q}_{\varkappa}[p]\right),
\end{equation}
where $Z_{(a,\xi)}(\widehat{Q}_{\varkappa}[p])=0$, and
\begin{equation}\label{lemma.5.proof.2}
Z_{(\xi,b)}(Q_{\varkappa}[p])\leqslant1+Z_{(\xi,b)}\left(\widehat{Q}_{\varkappa}[p]\right),
\end{equation}
where $Z_{(\xi,b)}(\widehat{Q}_{\varkappa}[p])=0$. Since $Q_{\varkappa}[p](\xi)\neq0$ and
$\widehat{Q}_{\varkappa}(\xi)=0$, we have
\begin{equation}\label{lemma.5.proof.3}
0=Z_{\{\xi\}}(Q_{\varkappa}[p])\leqslant-1+Z_{\{\xi\}}\left(\widehat{Q}_{\varkappa}[p]\right).
\end{equation}
Thus, summing the
inequalities~\eqref{lemma.5.proof.1}--\eqref{lemma.5.proof.3}, we
obtain~\eqref{main.work.formula.1}.

\vspace{2mm}

If $\widehat{Q}_{\varkappa}[p]$ has a unique zero $\xi$ in~$(a,b)$ and $Q_{\varkappa}[p](\xi)=0$, then,
by Lemma~\ref{lemma.4}, we have
\begin{equation*}\label{lemma.5.proof.4}
Z_{\{\xi\}}(Q_{\varkappa}[p])=1+Z_{\{\xi\}}\left(\widehat{Q}_{\varkappa}[p]\right),
\end{equation*}
and $Q_{\varkappa}[p](z)\neq0$ for $z\in(a,\xi)\cup(\xi,b)$. Therefore, the
inequality~\eqref{main.work.formula.1} is also true in this case.

\vspace{3mm}

Now, let $\widehat{Q}_{\varkappa}[p]$ have exactly $r\geqslant2$ \textit{distinct} real
zeroes, say $\xi_1<\xi_2<\ldots<\xi_r$, in the interval~$(a,b)$.
These zeroes divide $(a,b)$ into $r+1$ subintervals. If, for some
number $i$, $1\leqslant i\leqslant r$, $Q_{\varkappa}[p](\xi_i)\neq0$,
then by Lemma~\ref{lemma.3}, $Q_{\varkappa}[p]$ has \textit{at~most} one real
zero, counting multiplicity, in~$(\xi_{i-1},\xi_i]$
($\xi_0\stackrel{def}=a$). But~$\widehat{Q}_{\varkappa}[p]$ has \textit{at~least} one
real zero in $(\xi_{i-1},\xi_i]$, counting multiplicities (at the
point~$\xi_i$). Consequently,
\begin{equation}\label{intermediate.1.lemma.5}
Z_{(\xi_{i-1},\xi_i]}(Q_{\varkappa}[p])\leqslant Z_{(\xi_{i-1},\xi_i]}\left(\widehat{Q}_{\varkappa}[p]\right)
\end{equation}
If, for some number $i$, $1\leqslant i\leqslant r-1$, $Q_{\varkappa}[p](\xi_i)=0$
and $\xi_i$ is a zero of $\widehat{Q}_{\varkappa}[p]$ of multiplicity~$M$, then by
Lemma~\ref{lemma.4}, $Q_{\varkappa}[p]$ has only one zero $\xi_{i}$ of
multiplicity~$M+1$ in~$(\xi_{i-1},\xi_{i+1}]$. But in
the~interval~$(\xi_{i-1},\xi_{i+1}]$, $\widehat{Q}_{\varkappa}[p]$ has \textit{at
least} $M+1$ real zeroes, counting multiplicities (namely, $\xi_i$
which is a zero of multiplicity~$M$, and $\xi_{i+1}$). Therefore,
in this case, the following inequality holds
\begin{equation}\label{intermediate.2.lemma.5}
Z_{(\xi_{i-1},\xi_{i+1}]}(Q_{\varkappa}[p])\leqslant
Z_{(\xi_{i-1},\xi_{i+1}]}\left(\widehat{Q}_{\varkappa}[p]\right)
\end{equation}
Thus, if $Q_{\varkappa}[p](\xi_r)\neq0$, then
from~\eqref{intermediate.1.lemma.5}--\eqref{intermediate.2.lemma.5}
it follows that
\begin{equation}\label{intermediate.3.lemma.5}
Z_{(a,\xi_r]}(Q_{\varkappa}[p])\leqslant Z_{(a,\xi_r]}\left(\widehat{Q}_{\varkappa}[p]\right).
\end{equation}
But by Lemma~\ref{lemma.3}, $Q_{\varkappa}[p]$ has \textit{at~most} one real
zero, counting multiplicity, in the interval $(\xi_r,b)$.
Consequently, if $Q_{\varkappa}[p](\xi_r)\neq0$, then the
inequality~\eqref{main.work.formula.1} is valid.

If $Q_{\varkappa}[p](\xi_r)=0$, then by Lemma~\ref{lemma.4}, $Q_{\varkappa}[p](\xi_{r-1})\neq0$
(otherwise, $\xi_r$ cannot be a zero of $Q_{\varkappa}[p]$) and
from~\eqref{intermediate.1.lemma.5}--\eqref{intermediate.2.lemma.5}
it follows that
\begin{equation}\label{intermediate.4.lemma.5}
Z_{(a,\xi_{r-1}]}(Q_{\varkappa}[p])\leqslant Z_{(a,\xi_{r-1}]}\left(\widehat{Q}_{\varkappa}[p]\right).
\end{equation}
Now Lemma~\ref{lemma.4} implies
\begin{equation}\label{intermediate.5.lemma.5}
Z_{(\xi_{r-1},b)}(Q_{\varkappa}[p])=1+Z_{(\xi_{r-1},b)}\left(\widehat{Q}_{\varkappa}[p]\right),
\end{equation}
therefore, inequality~\eqref{main.work.formula.1} follows
from~\eqref{intermediate.4.lemma.5}--\eqref{intermediate.5.lemma.5}.
\end{proof}
%


\bibliographystyle{elsarticle-num}

\end{document}